\documentclass[11pt,a4paper]{amsart} 
\usepackage[T1]{fontenc}
\usepackage{amsmath,amsfonts,amssymb,graphicx,url,amsthm, mathrsfs}
\numberwithin{equation}{section}
\usepackage[left=2.5cm,right=2.5cm,top=2.5cm,bottom=2.5cm]{geometry}
\usepackage{enumitem,kantlipsum}
\usepackage{xcolor,color,nicefrac,svg}
\usepackage{graphicx, bbding}
\usepackage{booktabs}
\usepackage{csquotes}
\usepackage{caption}
\usepackage{subcaption}
\usepackage{mwe}
\usepackage{afterpage}
\usepackage[noend]{algpseudocode}
\usepackage[section,ruled]{algorithm}
\usepackage[export]{adjustbox}
\usepackage[sort]{natbib} 
\usepackage{setspace}
\usepackage[hyperfootnotes=true]{hyperref}
\hypersetup{colorlinks=true,raiselinks=true,breaklinks=true,citecolor=blue}
\usepackage{tikz,pgfplots,pgfplotstable}
\pgfplotsset{compat=newest}
\usetikzlibrary{arrows,decorations,backgrounds,positioning,calc,matrix}
\usepgfplotslibrary{groupplots}
\usetikzlibrary{external}
\tikzset{
    png export/.style={
        external/system call/.add={}{; convert -density 600 %
        -transparent white "\image.pdf" "\image.png"},
        /pgf/images/external info,
        /pgf/images/include external/.code={%
            \includegraphics
            [width=\pgfexternalwidth,height=\pgfexternalheight]
            {##1.png}%
        },
    },
    png export,
}

\pgfplotsset{compat=1.18}
\usepgfplotslibrary{colormaps} 
\pgfplotsset{colormap={viridis}{
  rgb=(0.2670,0.0049,0.3294)
  rgb=(0.2827,0.1002,0.4222)
  rgb=(0.2771,0.1852,0.4899)
  rgb=(0.2539,0.2653,0.5300)
  rgb=(0.2220,0.3392,0.5488)
  rgb=(0.1906,0.4071,0.5561)
  rgb=(0.1636,0.4711,0.5581)
  rgb=(0.1391,0.5338,0.5553)
  rgb=(0.1206,0.5964,0.5436)
  rgb=(0.1347,0.6586,0.5176)
  rgb=(0.2080,0.7187,0.4729)
  rgb=(0.3278,0.7740,0.4066)
  rgb=(0.4775,0.8214,0.3182)
  rgb=(0.6473,0.8584,0.2099)
  rgb=(0.8249,0.8847,0.1062)
  rgb=(0.9932,0.9062,0.1439)
}}
\pgfplotsset{colormap={coolwarm}{
  rgb=(0.2298,0.2987,0.7537)
  rgb=(0.3091,0.4135,0.8501)
  rgb=(0.3940,0.5224,0.9249)
  rgb=(0.4839,0.6220,0.9748)
  rgb=(0.5761,0.7088,0.9978)
  rgb=(0.6673,0.7792,0.9930)
  rgb=(0.7536,0.8302,0.9609)
  rgb=(0.8311,0.8595,0.9031)
  rgb=(0.8995,0.8475,0.8178)
  rgb=(0.9473,0.7947,0.7170)
  rgb=(0.9682,0.7208,0.6123)
  rgb=(0.9627,0.6282,0.5076)
  rgb=(0.9318,0.5191,0.4065)
  rgb=(0.8771,0.3946,0.3117)
  rgb=(0.8008,0.2508,0.2257)
  rgb=(0.7057,0.0156,0.1502)
}}

\pgfplotsset{
  colormap={turbo}{
    rgb=(0.18995,0.07176,0.23217)
    rgb=(0.25107,0.25237,0.63374)
    rgb=(0.27628,0.42592,0.81033)
    rgb=(0.19677,0.57878,0.86109)
    rgb=(0.12757,0.69912,0.82055)
    rgb=(0.13576,0.78442,0.71083)
    rgb=(0.25862,0.82031,0.52497)
    rgb=(0.47714,0.79137,0.25237)
    rgb=(0.74138,0.71773,0.07597)
    rgb=(0.97697,0.61849,0.02980)
    rgb=(0.99656,0.43960,0.14667)
    rgb=(0.94084,0.23127,0.15038)
    rgb=(0.86540,0.10419,0.13684)
    rgb=(0.75829,0.01514,0.12837)
  }}

\pgfplotsset{
  warpaxis/.style={
    width=5cm, axis equal image, enlargelimits=false, axis on top,
    xmin=0, xmax=1, ymin=0, ymax=1,
    xtick={0,0.5,1}, ytick={0,0.5,1},
    tick label style={font=\scriptsize},
    label style={font=\small},
  },
  warpcbar/.style={
  colorbar,
  colorbar style={
    width=2mm,
    tick label style={font=\scriptsize},
    yticklabel style={/pgf/number format/.cd, fixed},
  },
},
}

\newtheorem{theorem}{Theorem}[section]
\newtheorem{lemma}[theorem]{Lemma}
\newtheorem{corollary}[theorem]{Corollary}

\newtheorem{definition}[theorem]{Definition}
\newtheorem{example}[theorem]{Example}
\newtheorem{remark}[theorem]{Remark}

\usepackage{xparse}
\usepackage{verbatim}
\newsavebox{\fminipagebox}
\NewDocumentEnvironment{fminipage}{m O{\fboxsep}}
 {\par\kern#2\noindent\begin{lrbox}{\fminipagebox}
  \begin{minipage}{#1}\ignorespaces}
 {\end{minipage}\end{lrbox}%
  \makebox[#1]{%
    \kern\dimexpr-\fboxsep-\fboxrule\relax
    \fbox{\usebox{\fminipagebox}}%
    \kern\dimexpr-\fboxsep-\fboxrule\relax
  }\par\kern#2
 }

\def\letters{a,b,c,d,e,f,g,h,i,j,k,l,m,n,o,p,q,r,s,t,u,v,w,x,y,z}
\def\Letters{A,B,C,D,E,F,G,H,I,J,K,L,M,N,O,P,Q,R,S,T,U,V,W,X,Y,Z}
\makeatletter
\@for \@l:=\Letters \do{%
  \expandafter\edef\csname\@l bb\endcsname{\noexpand\ensuremath{%
  \noexpand\mathbb{\@l}}}%
  \expandafter\edef\csname\@l bf\endcsname{{\noexpand\bf \@l}}%
  \expandafter\edef\csname\@l cal\endcsname{\noexpand\ensuremath{%
  \noexpand\mathcal{\@l}}}%
  \expandafter\edef\csname\@l eu\endcsname{\noexpand\ensuremath{%
  \noexpand\EuScript{\@l}}}%
  \expandafter\edef\csname\@l frak\endcsname{\noexpand\ensuremath{%
  \noexpand\mathfrak{\@l}}}%
  \expandafter\edef\csname\@l rm\endcsname{{\noexpand\rm \@l}}%
  \expandafter\edef\csname\@l scr\endcsname{\noexpand\ensuremath{%
  \noexpand\mathscr{\@l}}}%
}
\@for \@l:=\letters \do{%
  \expandafter\edef\csname\@l bf\endcsname{{\noexpand\bf \@l}}%
  \expandafter\edef\csname\@l frak\endcsname{\noexpand\ensuremath{%
  \noexpand\mathfrak{\@l}}}%
  \expandafter\edef\csname\@l scr\endcsname{\noexpand\ensuremath{%
  \noexpand\mathscr{\@l}}}%
}
\makeatother
\definecolor{shadecolor}{rgb}{0.6, 0.6, 0.6} 
\definecolor{red}{rgb}{1,0,0.2} 
\definecolor{darkgreen}{rgb}{0, 0.6, 0}

\newcommand{\isdef}{\mathrel{\mathrel{\mathop:}=}}
\newcommand{\defis}{\mathrel{=\mathrel{\mathop:}}}

\newcommand{\R}{\mathbb{R}}
\newcommand{\N}{\mathbb{N}}

\newcommand{\bs}{\boldsymbol}
\DeclareMathOperator{\diam}{diam}

\DeclareMathOperator{\spn}{span}

\DeclareMathOperator*\argmin{\arg\min}

\DeclareMathOperator\dist{dist}
\DeclareMathOperator\supp{supp}

\renewcommand{\mid}{\hspace*{0.2ex}|\hspace*{0.2ex}}

\newcommand{\myatop}[2]{\genfrac{}{}{0pt}{}{#1}{#2}}

\newcommand{\xx}{\boldsymbol{x}}

\newcommand{\aalpha}{\boldsymbol{\alpha}}

\begin{document}
\title{Samplet compression for conditionally positive definite kernels and universal Kriging}
 \thanks{SA and MM gratefully acknowledge support from the 
   Swiss National Science
Foundation (SNSF)
for the financial support through
the SNSF starting grant ``Multiresolution
methods for unstructured data'' (TMSGI2\_211684).}
 \author{Sara Avesani}
\address{
Sara Avesani}
\email{sara.avesani@usi.ch} 
 \author{R\"udiger Kempf}
\address{
R\"udiger Kempf}
\email{ruediger.kempf@uni-bayreuth.de}
\author{Michael Multerer}
\address{
Michael Multerer}
\email{michael.multerer@usi.ch}
 \author{Holger Wendland}
\address{
Holger Wendland}
\email{holger.wendland@uni-bayreuth.de}


\begin{abstract}
We present a samplet-based framework for the efficient numerical solution of
saddle-point systems arising from conditionally positive definite (CPD) kernel
approximation in general and universal Kriging in particular. 
The vanishing moment
property of samplets as well as the particular structure of the associated scaling
distributions, which correspond to discrete orthogonal polynomials, 
allow for a numerically favorable representation of these
saddle-point systems. Concretely, they enable a natural null-space reduction of 
the indefinite saddle-point system to a (large) linear system for the detail 
coefficients and a small triangular system for the polynomial coefficients. 
We derive error 
bounds for the approximation by polyharmonic splines in Beppo-Levi spaces and show 
that the detail coefficients span precisely the subspace on which the CPD kernel is 
positive definite, rendering the reduced block symmetric positive definite. 
In view of the quasi-sparsity of the samplet-transformed
kernel matrix for asymptotically smooth kernels, the resulting method achieves 
$\mathcal{O}(N\log N)$ cost for the assembly and the storage of the saddle-point system. 
The reduced system can efficiently be solved by a sparse Cholesky factorization. 
We illustrate the framework with three applications, namely Gaussian process 
regression with {generalized} covariances, landmark-based image registration via 
samplet-compressed thin plate splines, and three-dimensional mesh deformation.
\end{abstract}

\maketitle
\section{Introduction}

\subsection{Motivation}
Kernel methods are the de facto go-to method for reconstructing
multi-dimensional functions from scattered data \cite{Schaback2006}. Most often,
the choice is to use \emph{positive definite} kernels since they lead, for any
configuration of data sites, to unique interpolants that are provably best
approximations or optimal recoveries of the target function in the native space
of the chosen kernel, \cite[Chapter 13]{Wendland_2004}. Kernel methods can be
seen as the extension of univariate spline theory. Historically however, this
extension is not done by positive definite kernels, but by \emph{conditionally
positive definite} kernels, in particular in connection with \emph{radial basis
functions} \cite{Duchon1976,Duchon1977,Duchon1978,Madych1990,Micchelli1986}.

A key motivation for employing conditionally positive definite (CPD) kernels
arises from the fact that approximants reproduce, e.g., low-degree polynomials
and naturally incorporate polynomial trends in the data. Consequently, CPD-based
schemes often achieve excellent approximation quality without the need to tune
shape parameters, in contrast to usual positive definite kernels. This
parameter-free nature is particularly attractive in large-scale numerical
settings where parameter selection can otherwise dominate the computational
effort. Furthermore, there is no clear black-box algorithm to obtain the best
shape parameter \cite{Fasshauer2007,Nassajian2023}.

Despite these advantages, the numerical treatment of CPD kernels presents
additional challenges compared to the positive definite case. The approximation
system involves a saddle-point structure, coupling the kernel matrix with
polynomial constraints, and is both dense and indefinite. For large data sets,
solving these augmented systems efficiently becomes a central computational
task. Developing fast numerical methods that exploit the structure of
CPD kernels, while maintaining stability and scalability, is therefore of
considerable practical importance.

\subsection{Literature review}
Typical examples of CPD functions are \emph{multiquadrics}, see 
\cite{Hardy1971,Hardy1990}, and \emph{thin plate splines}, see 
\cite{Duchon1976,Duchon1977,Duchon1978}. The latter can be seen as the immediate
extension of natural cubic splines in one dimension. These kernels are positive
definite on a certain, choosable finite-dimensional linear space, and essentially
let us use the approximation framework of positive definite kernels by retaining
the same numerical and analytical structure.

CPD kernels appear under different names across application domains. 
In geostatistics, and in particular in the context of universal Kriging,
the corresponding predictor is referred to as \emph{intrinsic Kriging},
where the generalized covariance matrix is conditionally positive definite;
see, for instance, \cite{Cressie1993,Matheron1973}.
In image analysis, landmark-based registration by thin plate splines goes back
to \cite{Bookstein1989}. In both cases, the classical exact formulation leads to
a dense, indefinite saddle-point system which requires $\mathcal{O}(N^3)$
computational work and $\mathcal{O}(N^2)$ memory.

Several approaches have been proposed to reduce this computational burden. One
line of work replaces the original global model with a sparse one: covariance
tapering and SPDE-based representations in geostatistics,
see \cite{Furrer2006,Lindgren2011}, and compactly supported radial basis
functions in image registration, see \cite{Fornefett2001}. A second line retains
the original kernel and accelerates the associated linear algebra using
hierarchical techniques, such as fast multipole methods and
$\mathcal{H}$-matrices,
see \cite{Beatson2001,Litvinenko2011}. 
Closely related, a multilevel basis with discrete vanishing moments has been used in
\cite{CastrillonCandas2013,CastrillonCandas2024}.
For an overview of
fast methods in this context, we refer to \cite{Beatson1997}. For the numerical
solution of saddle-point systems in general, we refer to \cite{BGL05}.

\subsection{The approach taken}
The setting of CPD approximation naturally allows for the use of
\emph{samplet matrix compression}, 
see \cite{Baroli2024,harbrecht2022samplets,Harbrecht2024}. Samplets are discrete
signed measures constructed such that all polynomials up to a certain,
prescribed degree vanish. This vanishing moment property allows the compression
of kernel matrices, demonstrated in \cite{Avesani2025,Harbrecht2025} for
positive definite kernels.

The purpose of this work is to introduce and demonstrate the use of samplet
compression for the saddle-point problem appearing for CPD kernels. Samplet
compression provides an algebraic alternative to the approaches above that
retains the exact CPD kernel while yielding a quasi-sparse representation
which can be compressed resulting in a cost of
$\mathcal{O}(N\log N)$ for matrix assembly and storage. Beyond compression,
the vanishing moment property of the samplets, together with the particular
structure of the associated scaling distributions, decouple the polynomial
constraints from the kernel part exactly at the discrete level. A null-space
approach, see \cite{BGL05}, then reduces the indefinite saddle-point system
to a system for the detail coefficients and a small triangular system for the
polynomial coefficients, where the detail coefficients span precisely the
subspace on which the CPD kernel is positive definite. The reduced system is
therefore symmetric positive definite and can efficiently be solved by a
sparse Cholesky factorization.

\subsection{Contributions}
The contributions of this article are the following.
\begin{itemize}
  \item We derive, for polyharmonic splines in Beppo-Levi spaces, error
    estimates for the regularized least-squares approximation that account
    explicitly for the observational noise.
\item We show that, in samplet coordinates, the polynomial constraint matrix
  of the CPD saddle-point system exhibits an exact block structure, with the
  polynomial information confined to the first few scaling distributions.
  \item Exploiting this structure, we reduce the indefinite saddle-point system
    to a small triangular system for the polynomial coefficients and a symmetric
    positive definite system for the detail coefficients, which spans exactly
    the subspace where the CPD kernel is positive definite.  
  \item We verify that polyharmonic splines and multiquadrics are
    asymptotically smooth, so that the samplet-transformed kernel matrix is
    quasi-sparse.
  \item We demonstrate the resulting method in three settings, namely universal
    Kriging on a large scattered data set, landmark-based image warping by thin
    plate splines, and three-dimensional mesh deformation, in each case at
    problem sizes that are not accessible to classical dense direct solvers.
\end{itemize}

\subsection{Outline}
The remainder of this article is organized as follows.
Section~\ref{sec:uk_CPD} collects the background on universal Kriging and 
CPD kernels, states the equivalence between CPD approximation and Kriging,
and derives the error estimates for regularized least-squares approximation
with polyharmonic splines. Section~\ref{sec:SampletsCPD} introduces samplets
and shows how the CPD saddle-point system decouples in samplet coordinates.
Section~\ref{sec:numerics} reports the numerical experiments in universal
Kriging and registration in two and three dimensions.
Section~\ref{sec:conclusion} summarizes the main findings of the
paper.

\section{Universal Kriging and conditionally positive definite kernels}\label{sec:uk_CPD}
\subsection{Universal Kriging}\label{subsec:UniversalKriging}
Let \(X=\{{\bs x}_1,\ldots,{\bs x}_N\}\subseteq D\) denote
a set of \emph{data sites} within a bounded Lipschitz domain
\( D\subset\Rbb^d\). Associated to the data sites, we consider the 
\emph{observations} \({\bs y}=[y_1,\ldots,y_N]^\intercal\in\Rbb^N\).
Consider the latent Gaussian random field model
\begin{equation}\label{eq:uk-model}
  Z({\bs x}) = m({\bs x}) + G({\bs x}),
\end{equation}
where the expectation satisfies
\(m\in\Pcal_q
\isdef\operatorname{span}\{{\bs x}^{\bs\alpha}:{\bs\alpha}
\in\Nbb^d_0,\|{\bs\alpha}\|_1\leq q\}\)
and \(G\) is a centered Gaussian random field with covariance function
\(\Kcal\colon D\times D\to\Rbb\). We assume that the
observations adhere to the relation
\[
  y_i = Z({\bs x}_i) + \varepsilon_i,
  \quad \varepsilon_i \stackrel{\text{iid}}{\sim} \Ncal(0,\nu^2),
  \quad i=1,\dots,N,
\]
which is a standard model, e.g., in spatial geostatistics, see
\cite{DR07,Ste99}.

Given a basis \(p_1,\ldots,p_n\), \(n\isdef\dim\Pcal_q\), of \(\Pcal_q\)
as well as the \emph{basis of kernel translates}
\(\phi_i\isdef\Kcal({\bs x}_i,\cdot)\), \(i=1,\ldots,N\), we introduce
the \emph{feature vectors}
\[
  {\bs p}({\bs x})\isdef[p_1({\bs x}),\ldots,p_n({\bs x})]\quad\text{and}\quad
  {\bs k}({\bs x})\isdef[\phi_1({\bs x}),\ldots,\phi_N({\bs x})].
\]
Then, the posterior expectation of \(Z\) given \({\bs y}\) reads
\begin{equation}\label{eq:mean-as-interpolant}
  \mu({\bs x})\isdef\Ebb[Z\mid{\bs y}]({\bs x})
  ={\bs k}({\bs x}){\bs c}+{\bs p}({\bs x}){\bs d},
\end{equation}
where the vectors \({\bs c}\isdef[c_1,\ldots, c_N]^\intercal\in\Rbb^N\)
and \({\bs d}\isdef[d_1,\ldots,d_{n}]^\intercal\in\Rbb^n\) solve the saddle
point system
\begin{equation}\label{eq:saddlepoint}
  \begin{bmatrix}{\bs K}_\nu & {\bs P}
  \\ {\bs P}^\intercal &{\bs 0}\end{bmatrix}
  \begin{bmatrix}{\bs c}\\ {\bs d}\end{bmatrix}
  =\begin{bmatrix}{\bs y}\\ {\bs 0}\end{bmatrix}
\end{equation}
with
\[
{\bs K}_\nu\isdef({\bs K}+\nu^2{\bs I})
\in\Rbb^{N\times N}, \quad \text{and}
  \quad
  {\bs P}\isdef[{\bs p}({\bs x}_i)]_{i=1}^N\in\Rbb^{N\times n},
\]	
respectively, where $  {\bs K}\isdef[{\bs k}({\bs x}_i)]_{i=1}^N$
and $ \bs I\in\Rbb^{N\times N} $ is the identity matrix. 

Similarly, the posterior covariance of \(Z\) given \({\bs y}\) is
\begin{equation}\label{eq:Kriging-variance}
  \begin{aligned}
    C({\bs x},{\bs x}')&\isdef
  \Ebb\big[\big(Z({\bs x})-\mu({\bs x})\big)
    \big(Z({\bs x}')-\mu({\bs x}')\big)\big]\\
                       &= \Kcal({\bs x},{\bs x}')
  - \begin{bmatrix}{\bs k}({\bs x}) & {\bs p}({\bs x})\end{bmatrix}
  \begin{bmatrix}{\bs K}_\nu& {\bs P}\\ 
    {\bs P}^\intercal & {\bs 0}\end{bmatrix}^{-1}
    \begin{bmatrix}{\bs k}^\intercal({\bs x}')\\
    {\bs p}^\intercal({\bs x}')\end{bmatrix}.
  \end{aligned}
\end{equation}
{The derivation of \eqref{eq:mean-as-interpolant}
and \eqref{eq:Kriging-variance} is classical, see \cite[Chapter 1.5]{Ste99}
and \cite[Chapter 32]{wackernagel2003multivariate}. In the latter reference,
\(\Kcal\) is only required to be conditionally positive definite in the sense of
Definition~\ref{subsec:CPD-interp}. It is then no longer the covariance
function of \(Z\), but rather prescribes the variances
\({\bs c}^\intercal{\bs K}{\bs c}\) of the linear combinations
\(\sum_{i=1}^N c_i Z({\bs x}_i)\) satisfying \({\bs P}^\intercal{\bs c}={\bs 0}\).
This generalization is known as \emph{intrinsic Kriging} and results in the
same saddle point system \eqref{eq:saddlepoint}. The function \(\Kcal\) is then
referred to as a \emph{generalized covariance function}.}
The classical representation of \eqref{eq:Kriging-variance} found there is
retrieved by inserting the Cholesky factorization
\[
  \begin{bmatrix}{\bs K}_\nu & {\bs P}\\
          {\bs P}^\intercal & {\bs 0}\end{bmatrix}=
          \begin{bmatrix}{\bs I} & {\bs 0}\\[0.2em]{\bs P}^\intercal{\bs K}_\nu^{-1} & 
          {\bs I}\end{bmatrix}
  \begin{bmatrix}{\bs K}_\nu & {\bs 0}\\[0.2em]{\bs 0} & 
  {\bs S}\end{bmatrix}
  \begin{bmatrix}{\bs I} & {\bs K}_\nu^{-1}{\bs P}\\[0.2em]{\bs 0} & {\bs I}\end{bmatrix}
\]
with
the Schur complement \({\bs S_{\bs K_{\nu}}}\isdef-{\bs P}^\intercal \bs K_\nu ^{-1}{\bs P}\).

From the Cholesky factorization, it is evident that the matrix is invertible, 
whenever \(\nu^2>0\) and \({\bs P}\) has full column rank.
In this case, it is well known that the inverse is given by
\[
\begin{bmatrix}{\bs K}_\nu & {\bs P}\\
          {\bs P}^\intercal & {\bs 0}\end{bmatrix}^{-1}
  =
  \begin{bmatrix}
    {\bs K}_\nu^{-1}
    + {\bs K}_\nu^{-1}{\bs P}
      {\bs S_{\bs K_{\nu}}}^{-1}
      {\bs P}^\intercal{\bs K}_\nu^{-1}
    &
    -{\bs K}_\nu^{-1}{\bs P}
      {\bs S_{\bs K_{\nu}}}^{-1}
    \\[0.4em]
    -{\bs S}^{-1}_{\bs K_{\nu}}
      {\bs P}^\intercal{\bs K}_\nu^{-1}
    &
    {\bs S}^{-1}_{\bs K_{\nu}}
  \end{bmatrix}.
\]

For the existence and uniqueness of a solution to
\eqref{eq:saddlepoint}, the weaker concept of conditionally positive
definiteness is sufficient, which will be addressed in the next section.
\subsection{Conditionally positive definite kernel approximation}
\label{subsec:CPD-interp}
We now consider the posterior expectation \eqref{eq:mean-as-interpolant},
at first in the noiseless case \(\nu=0\), within the context of kernel
interpolation.

\begin{definition}
For $ n \in \N $ let $ \mathcal{P}\isdef \spn\{ p_1, \dots, p_n \} $ be 
a linear space of functions $ p_i\colon D \to \Rbb $. 
A symmetric function $ \Kcal\colon  D \times  D \to \Rbb $ 
is \emph{$\mathcal{P}$-conditionally positive semidefinite}, 
if, for any $ N \in \N $ and any pairwise distinct points 
$ \xx_1, \dots, \xx_N \in  D $, the kernel matrix 
\({\bs K}\isdef[\Kcal({\bs x}_i,{\bs x}_j)]_{i,j=1}^N\) satisfies
\begin{equation}\label{eq:condPosDef}
{\bs c}^\intercal{\bs K}{\bs c}\geq 0
\quad\text{for any }{\bs c}\in\ker{\bs P}^\intercal,
\end{equation}
where \({\bs P}\isdef[p_j(\xx_i)]_{\myatop{i=1,\ldots,N}{j=1,\ldots,n}}\).
The function $\Kcal $ is
\emph{$\mathcal{P}$-conditionally positive definite (CPD)}, 
if the inequality \eqref{eq:condPosDef} holds strictly, 
unless $ {\bs c} $ is zero.
\end{definition}

\begin{example}\label{ex:CPDKernels}
Famous examples of $ \Pcal_q $-CPD kernels are the \emph{polyharmonic splines}
\begin{align*}
    \Kcal(\xx,\bs y) = \begin{cases}\| \xx - \bs y \|_2^\ell, 
    &\quad \text{if } \ell \text{ is odd} \\ \| \xx - \bs y \|_2 ^\ell 
    \log\| \xx - \bs y \|_2, &\quad \text{if } \ell \text{ is even}  
  \end{cases}
\end{align*}
with {$ q = \lfloor \frac{\ell}{2} \rfloor + 1$} and the \emph{multiquadrics}
\begin{align*}
  \Kcal(\xx, \bs y) = (-1)^{ \lceil \beta \rceil}(c^2 + 
  \| \xx - \bs y \|_2^2)^{\beta}, 
    \quad c,\beta>0,\beta \notin \Nbb,
\end{align*}
with $ q = \lceil \beta \rceil-1$.
\end{example}

Similar to the case of positive definite kernels, 
we can associate a \emph{native space} $ \Ncal_{\Kcal}( D) $ to any
CPD kernel $ \Kcal $. To this end, we start from the linear space
\begin{equation}\label{eq:FiniteDualSpace}
\begin{aligned}
L_{\Pcal} ( D) \isdef 
\bigg\{ \lambda = \sum_{j=1}^{N} c_j \delta_{\xx_j}  : & 
\ c_1,\ldots, c_N \in \Rbb, \xx_1, \dots, \xx_N \in  D,\\
& N \in \Nbb 
\text{ and } \lambda(p) = 0 \text{ for all } p \in \Pcal \bigg\}
\subset[C(D)]'
\end{aligned}
\end{equation}
and equip it with the inner product 
\begin{align*}
\langle \lambda,\kappa\rangle_{\Kcal} \isdef 
\sum_{i=1}^{N} \sum_{j=1}^{M} c_i c_j' \Kcal(\xx_i, \bs y_j),\quad
\text{where }\lambda=\sum_{i=1}^{N} c_i \delta_{\xx_i},\ 
\kappa=\sum_{j=1}^{M} c_j' \delta_{{\bs y}_j}.
\end{align*}
The canonical norm on \(L_{\Pcal} ( D)\) is given by 
\(\|\lambda\|_\Kcal\isdef\sqrt{\langle \lambda,\lambda\rangle_{\Kcal}}\).
 
 \begin{definition}
 Let $ \Kcal\colon D\times D\to\Rbb $ be a $ \Pcal $-CPD 
 kernel. The \emph{native space} $ \Ncal_{\Kcal}( D)$ of $ \Kcal $ 
 is defined as
 \begin{align*}
    \Ncal_K( D) \isdef \big\{ f \in C( D)  :  
    | \lambda(f) | <\infty \text{ for all } 
  \lambda \in L_{\Pcal}( D)\text{ with }  \| \lambda \|_{\Kcal}=1\big\}.
 \end{align*}
 A semi-norm on $ \Ncal_{\Kcal}( D)$ is given by duality according to 
 \begin{align*}
| f |_{\Ncal_{\Kcal}( D)}\isdef \sup_{\lambda 
\in L_{\Pcal}( D) \setminus \{0\}} 
\frac{|\lambda(f)|}{\| \lambda \|_{\Kcal}}.
 \end{align*}
 \end{definition}

In particular, the polyharmonic spline allows for an identification of its 
native space with a classical Beppo-Levi space, see, e.g.,
\cite{Wendland_2004}.

\begin{theorem}
Let $ k\in \N $ be such that $ k - d \notin \{0,2,4,\dots\} $ if $ k $ is odd.
Let $\Kcal $ be a polyharmonic spline that is $ \Pcal_{q} $-conditionally 
positive definite, where $ q = k-1 $. Then the native space 
$ \Ncal_{\Kcal}( D) $ coincides with the Beppo-Levi space 
\begin{align*}
  B\!L_{k}( D) \isdef \big\{ f \in L_1^{\mathrm{loc}}( D) :
  D^{\aalpha} f \in L_2( D) \text{ for all } |\aalpha | = k\big\}
\end{align*}
equipped with the semi-norm 
\begin{align*}
    |f|_{B\!L_{k}( D)}\isdef \bigg[\sum_{\|\aalpha\|_1 = k} 
    \frac{k!}{\aalpha!} \| D^{\aalpha} f \|_{L_2( D)}^2 
  \bigg]^{\frac{1}{2}}.
\end{align*}
The semi-norms $ |\cdot|_{\Ncal_{\Kcal}( D)} $ and 
$ |\cdot|_{B\!L_{k}( D)}$ are equivalent.
\end{theorem}

We now turn our attention to the interpolation problem with CPD kernels. 
Concretely, we focus on the interpolation problem
\begin{align*}
&\text{Find $ s_{X} \in \operatorname{span}\{\phi_1,\ldots,\phi_N\} +
\Pcal $ such that}\\
&\qquad s_{X}(\xx_i) = 
y_i
\quad\text{for }i=1,\ldots,N.
\end{align*}

Obviously, if a solution exists to the interpolation problem, 
it can be written as 
\begin{equation}\label{eq:augsys}
s_{X}=\sum_{j=1}^N c_j \Kcal(\cdot, {\bs x}_j)+\sum_{k=1}^{
n} d_k p_k.
\end{equation}
The coefficients $ \bs c \in \Rbb^N$ and $ \bs d \in \Rbb^n $ 
can be computed by solving the saddle point system
\begin{equation}\label{eq:saddlesystem_noiseless}
\begin{bmatrix}{\bs K} & {\bs P}\\ {\bs P}^\intercal &{\bs 0}\end{bmatrix}
\begin{bmatrix}{\bs c}\\ {\bs d}\end{bmatrix}=\begin{bmatrix}{\bs y}\\ {\bs 0}
\end{bmatrix},
\end{equation}
which coincides with \eqref{eq:saddlepoint} in the noiseless case \(\nu=0\).

The existence and uniqueness of a solution of the linear system 
\eqref{eq:saddlesystem_noiseless} depends on the particular
choices of $ \Pcal $ and $ X $. 

\begin{definition} The set \(X=\{{\bs x}_1,\ldots,{\bs x}_N\}\subseteq  D\)
  is called 
\emph{$\Pcal$-unisolvent}, if the associated generalized Vandermonde matrix 
\({\bs P}=[p_j(\xx_i)]_{\myatop{i=1,\ldots,N}{j=1,\ldots,n}}\) 
has full column rank.
\end{definition}
The following theorem, which gives a criterion for the existence of a unique
solution to \eqref{eq:saddlesystem_noiseless} comes from \cite{Wendland_2004}.
\begin{theorem}
If $ X $ is $ \Pcal $-unisolvent and \(\Kcal\) is \(\Pcal\)-CPD,
then the linear system \eqref{eq:saddlesystem_noiseless} has a unique solution.  
\end{theorem}

If we deal with noisy data, as outlined in Subsection
\ref{subsec:UniversalKriging}, interpolation is not desired as it leads to
overfitting. In this case, the linear system \eqref{eq:saddlepoint} is
the first order condition of a regularized least-squares approximation.

For $ \nu\in\Rbb $, define the functional 
$ J_{\nu}\colon \Ncal_{\Kcal}( D) \to \R $
as 
\begin{align*}
J_{\nu}(s) \isdef \sum_{j=1}^{N} | y_j - s(\xx_j) |^2 
+ \nu^2 | s |_{\Ncal_{\Kcal}( D)}^2.
\end{align*}
The regularized least-squares (RLS) problem is then given as
\begin{align*}
&\text{Find $ s_{X,\nu} \in \Ncal_{\Kcal}( D) $ such that}\\
&\qquad s_{X,\nu} = 
\argmin_{s \in \Ncal_{\Kcal}( D)} J_{\nu}(s).
\end{align*}

The \emph{representer theorem}, see \cite{Wahba1990}, states that the RLS 
problem has a unique solution $s_{X,\nu}$ and it can be expressed as
\begin{equation}\label{eq:RLSsolution}
s_{X,\nu}=\sum_{j=1}^N c_j \Kcal(\cdot, {\bs x}_j)+\sum_{k=1}^{n} d_k p_k,
\end{equation}
where the vectors \({\bs c}\isdef[c_1,\ldots, c_N]^\intercal\in\Rbb^N\), and 
\({\bs d}\isdef[d_1,\ldots,d_{n}]^\intercal\in\Rbb^n\) solve the saddle
point system
\eqref{eq:saddlepoint}. 

The deterministic CPD interpolation problem described above
and the universal Kriging predictor
\eqref{eq:mean-as-interpolant} coincide for \(\nu  = 0\),
see \cite{Matheron1981}. 
This equivalence also holds for the noisy case 
$ \nu \neq 0 $, see \cite[Chapter 6]{ras_wil_06}.

\subsection{Error Estimation} \label{subsec:Error}
To formalize the approximation methods introduced above, we will use the linear
operators $ \Ical_X, \Qcal_{X,\nu}\colon \R^N \to \Ncal_{\Kcal}( D) $
representing interpolation and RLS approximation, respectively. They are given
by
\begin{align*}
\Ical_X(\bs y) = s_{X} \quad \text{and} \quad \Qcal_{X,\nu}(\bs y) = s_{X,\nu}
\end{align*}

The error analysis for interpolation with CPD kernels is well-known, see, e.g.,
\cite[Chapter 11.3]{Wendland_2004}. Particularly, one can expect exponential 
convergence \cite[Theorem 11.22]{Wendland_2004}, in the case of interpolation by
multiquadrics. On the other hand, RLS approximation with CPD kernels is less
studied. We will focus on polyharmonic splines because of the characterization
of their native space and take particular care of keeping the information on
the noise. The proofs presented here closely follow the arguments in
\cite{LeGia2026}. We start with some preliminary results.

\begin{lemma}\label{lem:PropertiesRLSApproximation}
Let $  D \subseteq \R^d $ be a bounded Lipschitz domain and let 
$ X=\{\xx_1, \dots, \xx_N\} \subseteq  D $ be a $ \Pcal_q $-unisolvent set
of data sites with $ q \in \N $. Let $ \Kcal $ be a CPD kernel of
$ B\!L_k( D) $ with \(k=q+1\). Assume that $ Z \in \Ncal_{\Kcal}( D)$ and let 
$ y_i = Z(\xx_i) + \varepsilon_i $, $i=1,\ldots, N $. Denote by 
$ \bs\varepsilon = [\varepsilon_1, \dots, \varepsilon_N]^\intercal$ 
the noise vector.
Then, there holds

\begin{align*}
    \|{\bs y} - \Qcal_{X,\nu}(\bs y) \|_{\ell_2(X)}^2 
    + \nu^2 | \Qcal_{X,\nu}( \bs y) |_{\Ncal_{\Kcal}( D)}^2 
    \leq \| \bs \varepsilon \|_{2}^2 + \nu^2 | Z |_{\Ncal_{\Kcal}( D)}^2.
\end{align*}
In particular, we have
\begin{equation}\label{eq:Qstability}
    |\Qcal_{X,\nu}(\bs y) |_{\Ncal_{\Kcal}( D)} \leq 
    \bigg( \frac{1}{\nu^2} \| \bs \varepsilon \|_{2}^2 
    + |Z|_{\Ncal_{\Kcal}( D)}^2 \bigg)^{\frac{1}{2}}.
  \end{equation}
\end{lemma}

\begin{proof}
Since $\Qcal_{X,\nu}(\bs y)$ is the unique minimizer of $ J_{\nu} $, we have 
\begin{align*}
\| {\bs y} - \Qcal_{X,\nu}(\bs y) \|_{\ell_2(X)}^2 
+ \nu^2 |\Qcal_{X,\nu}(\bs y) |_{\Ncal_{\Kcal}( D)}^2 
= J_{\nu}\big(\Qcal_{X,\nu}(\bs y)\big) \leq J_{\nu}(Z) 
= \| \bs \varepsilon \|_{2}^2 + \nu^2 |Z|_{\Ncal_{\Kcal}( D)}^2.
\end{align*}
\end{proof}
In case \({\bs\varepsilon}={\bs 0}\), Equation \eqref{eq:Qstability}
provides the stability of the projector \(\Qcal_{X,\nu}\) in the native space norm, i.e.,
\[
    |\Qcal_{X,\nu}(Z) |_{\Ncal_{\Kcal}( D)} \leq 
    |Z|_{\Ncal_{\Kcal}( D)}.
\]

In the following, we require a sampling inequality, taken from \cite{Arcangli2011}. 
It states an estimate in terms of the fill distance \begin{align*}
h_{X, D} \isdef \sup_{\xx \in  D} \min_{\xx_i \in X} \| \xx - \xx_i \|_2.
\end{align*}
of the set of sites $ X $.

\begin{theorem}\label{thrm:SamplingInequality}
Let $ D \subseteq \R^d$ be a bounded Lipschitz domain.
Let $p, r \in [1, \infty]$ 
and let $\gamma \isdef \max\{2, p, r\}$. Let $ k > d/2$ and set 
$\ell_0 \isdef k - d(1/2 - 1/r)_+$. Define
\begin{align*}
    \ell \isdef \begin{cases} \ell_0 & \text{if } r \in \N 
      \text{ and either } \ell_0 > 2 
    \text{ and } \ell_0 \in \N, \text{ or } r = 2, \\ 
    \lceil \ell_0 \rceil - 1 & \text{otherwise.} \end{cases}
\end{align*}
Then, there exist constants $h_0 > 0$ and $C > 0$ such that for all point sets 
$X\subseteq  D$ with fill distance 
$h_{X, D} \leq h_0$ and all $f \in B\!L_k( D)$, we have
\begin{align*}
    |f|_{W^s_r( D)} \leq C \big( h_{X, D}^{k - s - d(1/2 - 1/r)_+} 
    |f|_{B\!L_k( D)} + h_{X, D}^{d/\gamma - s} 
    \|f\|_{\ell_p(X)} \big)
\end{align*}
for all $0 \leq s \leq \ell$, where $s \in \N$ if $r = \infty$.
\end{theorem}

We are now in the position to state the error estimate.

\begin{theorem}\label{thrm:Error}
With the notation and assumptions of Lemma \ref{lem:PropertiesRLSApproximation}
and Theorem \ref{thrm:SamplingInequality}, let $\mu$ be given by 
\eqref{eq:mean-as-interpolant} for $ \bs y = \bs z + \bs \varepsilon$.
Then there exist constants $ C, h_0 > 0 $ such that for all $ X $ with 
$ h_{X, D} < h_0 $, all $ 1 \leq r \leq \infty $, $ \gamma = \max\{ 2,r \}$ 
and $ 0 \leq s \leq \ell $, where $ \ell $ is as in 
Theorem \ref{thrm:SamplingInequality}, the estimate 
\begin{align*}
  \mathbb{E}\big[| Z - \mu |_{W^s_r( D)} \big] &
\leq C \bigg(h_{X, D}^{k-s - d\left(\frac{1}{2} - \frac{1}{r} \right)_+} 
+ \nu h_{X, D}^{\frac{d}{\gamma} - s} \bigg) |Z|_{B\!L_k( D)}\\
&+ C \bigg( \frac{1}{{\nu}}
h_{X, D}^{k-s - d\left(\frac{1}{2} - \frac{1}{r} \right)_+} 
+ h_{X, D}^{\frac{d}{\gamma} - s} \bigg) 
\mathbb{E} \big[\|\bs \varepsilon\|_2 \big]
\end{align*}
holds.
\end{theorem}

\begin{proof}
The posterior expectation $ \mu $ satisfies 
$ \mu = \Qcal_{X,\nu}(\bs y) = \Qcal_{X,\nu}(\bs z) 
+ \Qcal_{X,\nu}(\bs \varepsilon)$ by linearity. This means we have 
a splitting
\begin{align}\label{eq:ErrorProofSplitting}
  | Z - \mu |_{W^s_r( D)} \leq | Z - \Qcal_{X,\nu}(\bs z) |_{W^s_r( D)} 
  + | \Qcal_{X,\nu}(\bs \varepsilon) |_{W^s_r( D)}.
\end{align}
The first term on the right-hand side can be bounded by  using Theorem
\ref{thrm:SamplingInequality} and in combination with Lemma 
\ref{lem:PropertiesRLSApproximation}. We obtain 
\begin{align}\label{eq:ErrorProof1}
| Z - \Qcal_{X,\nu}(\bs z)|_{W^s_r( D)} &
\leq C \left( h_{X, D}^{k-s - d (1/2 - 1/r)_+} 
| Z - \Qcal_{X,\nu}(Z)|_{B\!L_k( D)} + h_{X, D}^{d/\gamma - s} 
\| Z - \Qcal_{X,\nu}(\bs z)\|_{\ell_2(X)}\right) \nonumber \\
&\leq C h_{X, D}^{k-s - d (1/2 - 1/r)_+} | Z |_{B\!L_k( D)} 
+ C h_{X, D}^{d/\gamma - s}  \nu | Z |_{B\!L_k( D)}.
\end{align}
To finish the proof we have to control
$ | \Qcal_{X,\nu}(\bs \varepsilon) |_{W^s_r( D)} $. 
Again, with Lemma \ref{lem:PropertiesRLSApproximation}, we have 
\begin{align*}
| \Qcal_{X,\nu}(\bs \varepsilon) |_{B\!L_k( D)} \leq \frac{C}{\nu} 
\| \bs \varepsilon \|_2
\end{align*}
and 
\begin{align*}
\| \Qcal_{X,\nu}(\bs \varepsilon) \|_{\ell_2(X)}  
\leq \| \bs \varepsilon - \Qcal_{X,\nu}(\bs \varepsilon) \|_{\ell_2(X)} 
+ \| \bs \varepsilon \|_{2} \leq 2 \| \bs \varepsilon \|_{2}.
\end{align*}
Using the sampling inequality Theorem \ref{thrm:SamplingInequality}
for $ \Qcal_{X,\nu}(\bs \varepsilon) $ then yields
\begin{equation}\label{eq:ErrorProof2}
\begin{aligned}   
  |\Qcal_{X,\nu}(\bs \varepsilon) |_{W^s_r( D)} &
    \leq C h_{X, D}^{k - s - d(1/2 - 1/r)_+}
    |\Qcal_{X,\nu}(\bs \varepsilon) |_{B\!L_k( D)} 
    + Ch_{X, D}^{d/\gamma - s } \| \bs \varepsilon \|_{2} \\
    &\leq C \frac{1}{{\nu}} h_{X, D}^{k - s - d(1/2 - 1/r)_+}
    \|\bs \varepsilon\|_2 + Ch_{X, D}^{d/\gamma - s}  \| \bs \varepsilon \|_2.
\end{aligned}
\end{equation}
Inserting \eqref{eq:ErrorProof1} and \eqref{eq:ErrorProof2}
into \eqref{eq:ErrorProofSplitting} and taking the expected value concludes
the proof.
\end{proof}

In Subsection \ref{subsec:UniversalKriging} we assumed that the noise satisfied
$ \varepsilon \sim \Ncal(0,\nu^2) $. We can use this to refine the error 
estimate in Theorem \ref{thrm:Error}. To this end, we introduce the
\emph{separation radius} $ q_X $ of $ X $ as
\begin{align*}
q_X \isdef \frac{1}{2} \min_{i \neq j} \| \xx_i - \xx_j \|_2
\end{align*}
and the \emph{mesh ratio} $\rho_{X, D}$ of $ X $ in $  D $ as
\begin{align*}
    \rho_{X, D} \isdef \frac{h_{X, D}}{q_X}.
\end{align*}

\begin{corollary}
With the notation and assumptions of Theorem \ref{thrm:Error} assume that the
noise satisfies $ \varepsilon_i \sim \Ncal(0,\nu^2) $, $ i = 1, \dots, N $.
Then the estimate 
\begin{align*}
  \mathbb{E}\big[| Z - \mu |_{W^s_r( D)} \big] 
&\leq C \left(h_{X, D}^{k-s - d\left(\frac{1}{2} - \frac{1}{r} \right)_+} 
+ \nu h_{X, D}^{\frac{d}{\gamma} - s} \right) |Z|_{B\!L_k( D)}\\
&\qquad+ C \rho_{X, D}^{\frac{d}{2}} \left( h_{X, D}^{k-s - d\left(\frac{1}{2} 
- \frac{1}{r} \right)_+ - \frac{d}{2}} + \nu h_{X, D}^{\frac{d}{\gamma} 
- s - \frac{d}{2}} \right) 
\end{align*}
holds.
\end{corollary}

\begin{proof}
We have $ N \leq C q_X^{-d} = C h_{X, D}^{-d} \rho_{X, D}^d $ and 
\begin{align*}
  \mathbb{E}[\| \bs \varepsilon \|_2] \leq 
  \bigg( \mathbb{E}\bigg[ \sum_{i=1}^{N} | \varepsilon_i |^2\bigg]
  \bigg)^\frac{1}{2} =
  \sqrt{N} \nu \leq C h_{X, D}^{-\frac{d}{2}} \rho_{X, D}^{\frac{d}{2}} \nu.
\end{align*}
Inserting this into the bound of Theorem \ref{thrm:Error} completes the proof.
\end{proof}

\section{Conditionally Positive Definite Kernels in Samplet Coordinates}
\label{sec:SampletsCPD}
\subsection{Samplets}
Samplets are discrete signed measures, which exhibit vanishing moments. 
They can be constructed such that they are contained in \(L_{\Pcal} ( D)\), 
see \eqref{eq:FiniteDualSpace}, for a desired set of primitives \(\Pcal\).
We give a brief introduction on the underlying concepts and refer to
\cite{harbrecht2022samplets} for all the details. 
Given the set of data sites \(X\), we define the space
\begin{equation} \label{eq:spaceS}
\Xcal'\isdef\operatorname{span}\{\delta_{\bs x}:{\bs x}\in X\}
\subset[C(D)]'
\end{equation}
equipped with the inner
product \(\langle\delta_{{\bs x}_i},
\delta_{{\bs x}_j}\rangle_{\Xcal'}\isdef\delta_{ij}\)
for \({\bs x}_i,{\bs x}_j\in X\),
which is different from the native space one, 
compare \cite{balazs2024construction}.

We introduce a multiresolution analysis
\(
\Xcal_{0}'\subset\Xcal_{1}'\subset\cdots\subset\Xcal_{J}'=\Xcal',
\)
which keeps track of the increment of information between consecutive 
levels $j$ and $j+1$. 
Since $\Xcal_{j}'\subset \Xcal_{j+1}'$, we can
orthogonally decompose 
\(
\Xcal_{j+1}' = \Xcal_{j}'{\oplus}\Scal_{j}'
\)
for a certain \emph{detail space} $\Scal_{j}'\perp \Xcal_{j}'$. In analogy to wavelet
nomenclature, we call the elements of a basis of \(\Xcal_{0}'\) 
\emph{scaling distributions} and the elements of a basis of any of the spaces
$\Scal_{j}'$ \emph{samplets}. The collection of the bases of $\Scal_{j}'$ for
\(j=0,\ldots,J-1\)\/ together with a basis of \(\Xcal_{0}'\) is called a
\emph{samplet basis} for \(\Xcal'\). A samplet basis can be constructed such
that it exhibits \emph{vanishing moments} with respect to general primitives
\(\Pcal\), see again \cite{balazs2024construction}, i.e.,
\begin{equation}\label{eq:vanishingMoments}
\sigma_{j,k}(p)
 = 0\quad\text{for all}\ p\in\Pcal\text{ and any }\sigma_{j,k}\in \Scal_{j}'.
\end{equation}
A canonical choice, which we make in the following, is \(\Pcal=\Pcal_q\), 
that is, the space of polynomials up to total degree \(q\).

The samplet construction for \(\Xcal'\) is based on a 
\emph{hierarchical cluster tree} for
the set of data sites \(X\), i.e., a tree $\mathcal{T}$ with root \(X\) such
that each node \(\tau\in\Tcal\) is the disjoint union of its children.
A hierarchical cluster tree \(\Tcal\) for \(X\) amounts to
a support based clustering of the
Dirac-$\delta$-distributions spanning \(\Xcal'\).
With a slight abuse of notation, we will refer to this cluster tree also by
\(\Tcal\).
Given a balanced cluster tree \(\Tcal\), a samplet basis for \(\Xcal'\)
can be constructed with cost \(\Ocal(N)\). In case of a balanced
binary tree, the maximum level satisfies
\(J=\lceil\log_2(N)\rceil\) and the samplet basis exhibits the 
following properties,
see \cite{harbrecht2022samplets, avesani2025multiresolution} for details. 
\begin{theorem}\label{theo:waveletProperties}
The samplet basis \(\bigcup_{j=0}^{J}\{\sigma_{j,k}\}_k\)
forms an orthonormal basis in $\Xcal'$, satisfying
the following properties:
\begin{enumerate}[label=(\roman*)]
\item 
There holds $\operatorname{dim}\Xcal'_j\sim 2^{j}$.
\item 
The samplet basis exhibits \emph{vanishing moments} of order \(q+1\), i.e.,
\(
\sigma_{j,k}(p)
 = 0\ \text{for all}\ p\in\Pcal_q,\ j\geq 1.
\)
\item Let ${\bs x}_0 \in  D$, $f \in C^\alpha({\bs x}_0)$ for some $\alpha \geq 0$
with $q \geq \lfloor \alpha \rfloor$. Then, for every cluster 
$\tau_{j,k}\in\Tcal$ containing ${\bs x}_0$, there holds
\(
|\sigma_{j,k}(f)| \lesssim \operatorname{diam}(\tau_{j,k})^\alpha
\sqrt{\#\tau_{j,k}},
\)
where $C^\alpha(x_0)$ denotes the Jaffard regularity at ${\bs x}_0$. 
\end{enumerate}
\end{theorem}

The orthonormality of the samplet basis yields that the
samplet transform
\(
[\sigma_{j,k}]_{j,k}={\bs T}[\delta_{{\bs x}_i}]_{i=1}^{N}
\)
satisfies \({\bs T}^\intercal{\bs T}={\bs T}{\bs T}^\intercal
={\bs I}\in\Rbb^{N\times N}\). 
Moreover, the cost of the samplet transform
is of order \(\Ocal(N)\),
if it is computed recursively starting from the leaves
of the cluster tree \(\Tcal\).

Since \(\Xcal'\subseteq\Ncal_{\Kcal}'( D)\),
employing the Riesz isometry, the samplet basis induces a basis
for a subspace \(\Xcal\subseteq\Ncal_{\Kcal}( D)\).
Concretely, the samplet 
\(
\sigma_{j,k}=\sum_{i=1}^{N} \omega_{j,k,i}\delta_{{\bs x}_{i}}
\)
can be identified with the function
\(
\psi_{j,k}\isdef
\sum_{i=1}^{N} \omega_{j,k,i}\Kcal({\bs x}_i,\cdot).
\)
The vanishing moment property 
\eqref{eq:vanishingMoments} then translates to
\(
\langle\psi_{j,k},f\rangle_{\Kcal}=0
\)
for any \(f\in\Ncal_\Kcal(D)\) which satisfies 
\(f|_{\supp(\sigma_{j,k})}\in\Pcal\). 

Furthermore, there holds
\begin{equation} \label{eq:K_sigma}
\big[\big\langle\psi_{j,k},\psi_{j',k'}\big\rangle_{\Kcal}\big]_{j,j',k,k'}
= {\bs T}{\bs K}{\bs T}^\intercal\defis {\bs K}^{\Sigma},
\end{equation}
which means that the Gramian of the samplet basis coincides with the
samplet transformed kernel matrix.
For \emph{asymptotically smooth} kernels \(\Kcal\), i.e.,
there exist $C,r>0$ such that
\begin{equation}\label{HM_eg:kernel_estimate}
  \bigg|\frac{\partial^{|\bs\alpha|+|\bs\beta|}}
  	{\partial{\bs x}^{\bs\alpha}
  	\partial{\bs z}^{\bs\beta}} \Kcal({\bs x},{\bs z})\bigg|
  		\leq C \frac{(|\bs\alpha|+|\bs\beta|)!}
		{r^{|\bs\alpha|+|\bs\beta|}
		\|{\bs x}-{\bs z}\|_2^{|\bs\alpha|+|\bs\beta|}}
\end{equation}
uniformly in $\bs\alpha,\bs\beta\in\mathbb{N}^d$ for all 
\({\bs x},{\bs z}\in D\)
with \({\bs x}\neq{\bs z}\), 
the matrix \( {\bs K}^{\Sigma}\) becomes
quasi-sparse and can be efficiently computed.

We have the following statement from \cite{HM24}.

\begin{theorem}\label{thm:compression}
Let \(X\) be quasi-uniform and
set all coefficients of the matrix ${\bs K}^\Sigma$ from 
\eqref{eq:K_sigma} to zero which satisfy the admissibility condition

\begin{equation}\label{HM_eg:cutoff}
   \dist(\tau,\tau')\ge\rho\max\{\diam(\tau),\diam(\tau')\},\quad\rho>0,
\end{equation}
where \(\tau\) is the cluster supporting \(\sigma_{j,k}\) and \(\tau'\) is the 
cluster supporting \(\sigma_{j',k'}\), respectively. 
{Then, there exists
a constant \(C>0\), such that the resulting compressed matrix 
${\bs K}^{\Sigma,\rho}$ satisfies}
\begin{equation}\label{eq:CompressionError}
 {\big\|{\bs K}^\Sigma-{\bs K}^{\Sigma,\rho}\big\|_F}
   \leq C \bigg(\frac{r\rho}{d}\bigg)^{-2(q+1)}{\big\|{\bs K}^\Sigma\big\|_F}.
\end{equation}
The compressed matrix has $\mathcal{O}(N\log N)$ nonzero
coefficients.
\end{theorem}

For the sake of completeness, we show that the kernels of 
Example~\ref{ex:CPDKernels} satisfy \eqref{HM_eg:kernel_estimate}. To this end, 
we need an estimate on derivatives of radial functions. The key ingredient will
be a specific version of Faa di Bruno's formula for radial functions and the 
proof that leads to it, see, e.g., \cite[Proposition 3.5]{Franz2018}.

\begin{theorem}\label{thrm:FaaDiBruno}
Let $ f \in C^s([0,\infty))$, $ s \in \N $ and let $ \aalpha \in \N^d_0 $ with 
$ 1 < |\aalpha| \leq s $. Then, there are certain coefficients 
$ c_{\bs{\beta}} \in \R $, $ \bs{\beta} \in \N^d_0 $, such that 
\begin{align*}
\partial^{\aalpha}f(\|\xx\|_2) = \sum_{m=1}^{|\aalpha|} f^{(m)}(\|\xx\|_2) 
\sum_{\ell =1}^{2|\aalpha|-1} \|\xx\|_2^{-\ell} \sum_{|\bs{\beta}| 
= \ell + m - |\aalpha|} c_{\bs{\beta}} \xx^{\bs{\beta}}, 
\quad \xx \in\R^d \setminus\{ \bs{0} \}.
\end{align*}
\end{theorem}

\begin{remark}
Taking a close look at the proof of \cite[Proposition 3.5]{Franz2018}, a technical induction yields the bound 
\begin{align}
\sum_{\ell =1}^{2|\aalpha|-1}  \sum_{|\bs{\beta}| = \ell + m - |\aalpha|}
|c_{\bs{\beta}}| \leq 4^{|\aalpha|} |\aalpha|!, \quad \text{for all } 1 \leq m \leq | \aalpha |.
\end{align}
\end{remark}

We can now prove the main estimate used to show that the kernels of interest are indeed asymptotically smooth.

\begin{corollary}\label{cor:BoundOnDerivative}
In the context of Theorem \ref{thrm:FaaDiBruno}, there holds 
\begin{align*}
|\partial^{\aalpha} f(\| \xx \|_2) | \leq 4^{|\aalpha|} 
| \aalpha|! \sum_{m=1}^{|\aalpha|}\big| f^{(m)}(\| \xx \|_2)\big| 
\| \xx \|_2^{- |\aalpha| + m}, \quad \xx \neq \bs{0}.
\end{align*}
\end{corollary}

\begin{proof}
We prove the claim by induction over $ n \isdef |\aalpha| $. For $ n = 1 $, 
we have $ \aalpha = \bs{e}_j $, $ 1 \leq j \leq d $, and there holds
\begin{align*}
\frac{\partial}{\partial x_j} f(\| \xx \|_2) = f'(\| \xx \|_2) 
\frac{x_j}{\| \xx \|_2}.
\end{align*} 
This implies
\begin{align*}
\bigg| \frac{\partial}{\partial x_j} f(\| \xx \|_2) \bigg| 
\leq| f'(\| \xx \|_2) |\leq 4 | f'(\| \xx \|_2) |,
\end{align*}
which is the statement. For the induction step, let the induction hypothesis
be valid for $ |\bs\alpha|=n $ with $ n \geq 1$. 
We can write each multi-index of modulus 
$ n+1 $ as $ \aalpha + \bs{e}_j $ with a multi-index $ \aalpha $ of modulus 
$ n $ and $ 1 \leq j \leq d $. There holds the following representation of 
$ \partial^{\aalpha + \bs{e}_j}f(\| \xx \|_2)$, 
see \cite[Proof of Proposition 3.5]{Franz2018},
\begin{align*}
\partial^{\aalpha + \bs{e}_j}f(\| \xx \|_2) &= \sum_{m=1}^{|\aalpha|} 
\frac{x_j}{\|\xx\|_2} f^{(m+1)}(\|\xx\|_2) \sum_{\ell =1}^{2|\aalpha|-1} 
\|\xx\|_2^{-\ell} \sum_{|\bs{\beta}| = \ell + m - 
|\aalpha|} c_{\bs{\beta}} \xx^{\bs{\beta}} \\
&\phantom{=}- \sum_{m=1}^{|\aalpha|} f^{(m)}(\|\xx\|_2) 
\sum_{\ell =1}^{2|\aalpha|-1} 
\ell \|\xx\|_2^{-\ell-2 } \sum_{|\bs{\beta}| = \ell + m - |\aalpha|} 
c_{\bs{\beta}} \xx^{\bs{\beta} + \bs{e}_j} \\
&\phantom{=}+ \sum_{m=1}^{|\aalpha|} f^{(m)}(\|\xx\|_2) 
\sum_{\ell =1}^{2|\aalpha|-1} \|\xx\|_2^{-\ell} \sum_{|\bs{\beta}| 
= \ell + m - |\aalpha|} 
c_{\bs{\beta}} \beta_j \xx^{\bs{\beta} - \bs{e}_j}.
\end{align*} 
Taking the absolute value on both sides and exploiting the upper bound
$ | \xx^{\bs{\beta}} | \leq \|\xx\|_2^{|\bs{\beta}|}$, yields 
\begin{equation}\label{eq:ProofBoundDerivative1}
\begin{aligned}\big| \partial^{\aalpha + \bs{e}_j}f(\| \xx \|_2) \big| & 
\leq \sum_{m=1}^{|\aalpha|} \big| f^{(m+1)} (\| \xx \|_2) \big| 
\| \xx \|_2^{-|\aalpha| + m} \sum_{\ell =1}^{2|\aalpha|-1} \sum_{|\bs{\beta}| 
= \ell + m - |\aalpha|} | c_{\bs{\beta}} |\\
&\phantom{=} + \sum_{m=1}^{|\aalpha|} \big| f^{(m)} (\| \xx \|_2) \big| 
\| \xx \|_2^{-|\aalpha| + m-1} \sum_{\ell=1}^{2|\aalpha|-1} \sum_{|\bs{\beta}| 
= \ell + m - |\aalpha|} \ell | c_{\bs{\beta}} |\\
&\phantom{=} + \sum_{m=1}^{|\aalpha|} \big| f^{(m)} (\| \xx \|_2) \big| 
\| \xx \|_2^{-|\aalpha| + m-1} \sum_{\ell =1}^{2|\aalpha|-1} \sum_{|\bs{\beta}| 
= \ell + m - |\aalpha|}  \beta_j  |c_{\bs{\beta}} |.
\end{aligned}
\end{equation}

Next, we bound the innermost double sums in the three terms on the right-hand
side. In the first term, we directly use the induction hypothesis while the
second term is estimated by
\begin{align*}
\sum_{\ell =1}^{2|\aalpha|-1} \sum_{|\bs{\beta}| 
= \ell + m - |\aalpha|} \ell | c_{\bs{\beta}} | 
\leq (2 | \aalpha | - 1) \sum_{\ell =1}^{2|\aalpha|-1} \sum_{|\bs{\beta}| 
= \ell + m - |\aalpha|} | c_{\bs{\beta}} | 
\leq (2 | \aalpha | - 1) 4^{| \aalpha |} | \aalpha |!.
\end{align*}
Finally, for the third term, we have 
\begin{align*}
\sum_{\ell =1}^{2|\aalpha|-1} \sum_{|\bs{\beta}| = \ell + m - |\aalpha|} 
 \beta_j |c_{\bs{\beta}} | &\leq \sum_{\ell =1}^{2|\aalpha|-1}
 (\ell + m - | \aalpha |)
 \sum_{|\bs{\beta}| = \ell + m - |\aalpha|} |c_{\bs{\beta}} |
 \leq (2 | \aalpha | - 1) 4^{| \aalpha |} | \aalpha |! 
\end{align*}
where we used that $1 \leq m \leq | \aalpha | $.

Inserting these bounds into \eqref{eq:ProofBoundDerivative1} yields
\begin{align*}
\big|\partial^{\aalpha + \bs{e}_j}f(\| \xx \|_2) \big| & 
\leq \sum_{m=1}^{|\aalpha|} \left| f^{(m+1)} (\| \xx \|_2) \right|
\| \xx \|_2^{-|\aalpha| + m} 4^{|\aalpha|}|\aalpha|! \\
&\phantom{=} + \sum_{m=1}^{|\aalpha|} \left| f^{(m)} (\| \xx \|_2) \right|
\| \xx \|_2^{-|\aalpha| + m -1}(2 | \aalpha | - 1) 4^{| \aalpha |} | \aalpha |!\\
&\phantom{=} + \sum_{m=1}^{|\aalpha|} \left| f^{(m)} (\| \xx \|_2) \right|
\| \xx \|_2^{-|\aalpha| + m-1} (2 | \aalpha | - 1) 4^{| \aalpha |} | \aalpha |!.
\end{align*}
From this, we infer by shifting the index in the first sum
and appropriately modifying the last two summands that
\begin{align*}
\big|\partial^{\aalpha + \bs{e}_j}f(\| \xx \|_2) \big| & 
\leq \sum_{m=1}^{|\aalpha|+1} \big| f^{(m)} (\| \xx \|_2) \big|
\| \xx \|_2^{-|\aalpha| + m-1} (4|\bs\alpha|-1)4^{|\aalpha|}|\aalpha|! \\
&\leq 4^{|\bs\alpha|+1}(|\bs\alpha|+1)!\sum_{m=1}^{|\aalpha|+1} \big| 
f^{(m)} (\| \xx \|_2) \big|
\| \xx \|_2^{-(|\aalpha|+1) + m}, 
\end{align*}
as claimed.
\end{proof}

This helps us to prove that the kernels introduces in Example~\ref{ex:CPDKernels} 
satisfy \eqref{HM_eg:kernel_estimate}.

\begin{corollary}
The polyharmonic splines and the multiquadrics are asymptotically smooth.
\end{corollary}

\begin{proof}
We only show the claim for the polyharmonic spline in the case that $ k $ is
odd. All other statements can be proven the same way. Since the kernel is
radial, we use Corollary~\ref{cor:BoundOnDerivative} with $ f(r) = r^\ell $ for
fixed, odd $ \ell $. 

There holds 
\[
f^{(m)}(r) = \frac{\ell!}{(\ell-m)!} r^{\ell-m},\] 
which yields
$| f^{(m)}(r) | \leq \ell! r^{\ell-m} $. Inserting this into the bound from
Corollary~\ref{cor:BoundOnDerivative}, we obtain 
\begin{align*}
  |\partial^{\bs{\alpha}} f(\|{\bs x}\|_2) | &\leq 4^{|\bs{\alpha}|} 
| \bs{\alpha} |!\sum_{m=1}^{|\bs\alpha |} \ell!\|{\bs x}\|_2^{\ell-m}
\|{\bs x}\|_2^{- | \bs{\alpha}| + m} 
= 
4^{|\bs{\alpha}|} 
| \bs{\alpha} |! \|{\bs x}\|_2^{\ell-| \bs{\alpha}| }
\sum_{m=1}^{|\bs\alpha |} \ell!\\
&= 4^{|\bs{\alpha}|} 
| \bs{\alpha} |! \|{\bs x}\|_2^{\ell-| \bs{\alpha}| }|\bs\alpha|\ell!
\leq 8^{|\bs\alpha|}\ell!| \bs{\alpha} |! \|{\bs x}\|_2^{\ell-| \bs{\alpha}| },
\end{align*}
for any multi-index $ \bs{\alpha} \in \N^d_0$.
Since $  D $ is assumed to be bounded, 
we have $ \|{\bs x}-{\bs y}\|_2 \leq \operatorname{diam}( D) $
for all \({\bs x},{\bs y}\in D\), which yields
\begin{align*}
\big|\partial^{\aalpha}\partial^{\bs{\beta}} 
\Kcal(\xx, \bs{y})\big| \leq 
8^{|\bs\alpha+\bs\beta|}\ell! 
\operatorname{diam}( D)^\ell 8^{|\bs{\alpha}+{\bs\beta}|} 
(|\bs{\alpha}|+|\bs\beta|)! \| \xx - \bs{y} \|_2^{- (|\bs{\alpha}|+|\bs\beta|)}.
\end{align*}
Hence, the polyharmonic spline with $ \ell $ odd is asymptotically smooth 
with $ C = \ell! \operatorname{diam}( D)^\ell$ and 
$ r = 1/8$, compare \eqref{HM_eg:kernel_estimate}.
\end{proof}

\subsection{Samplet representation of conditionally positive definite kernels}
We now consider the samplet representation of the linear system
\eqref{eq:saddlepoint} for an asymptotically smooth kernel \(\Kcal\).
Applying the partial samplet transform
\[
\begin{bmatrix}{\bs T} & \\ & {\bs I}\end{bmatrix}\in\Rbb^{(N+n)\times (N+n)}
\]
to \eqref{eq:saddlepoint}, we obtain
\[
\begin{bmatrix}{\bs T} & \\ & {\bs I}\end{bmatrix}
\begin{bmatrix}{\bs K}_\nu& {\bs P}\\ {\bs P}^\intercal &{\bs 0}\end{bmatrix}
\begin{bmatrix}{\bs T} & \\ & {\bs I}\end{bmatrix}^\intercal
\begin{bmatrix}{\bs T} & \\ & {\bs I}\end{bmatrix}
\begin{bmatrix}{\bs c}\\ {\bs d}\end{bmatrix}=
\begin{bmatrix}{\bs T} & \\ & {\bs I}\end{bmatrix}\begin{bmatrix}{\bs y}\\ 
{\bs 0}
\end{bmatrix}.
\]
This equation can be simplified towards
\begin{equation}\label{trafProb}
\begin{bmatrix}{\bs T}{\bs K}_\nu{\bs T}^\intercal & {\bs T}{\bs P}\\ 
{\bs P}^\intercal{\bs T}^\intercal &{\bs 0}\end{bmatrix}
\begin{bmatrix}{\bs T}{\bs c}\\ {\bs d}\end{bmatrix}=
\begin{bmatrix}{\bs T}{\bs y}\\ {\bs 0}.
\end{bmatrix}
\end{equation}

In \cite{GMQ26}, it has been proven that the scaling distributions correspond 
to the evaluation of the discrete orthogonal polynomials with respect to the 
sample measure at the data sites evaluated at the very same data sites. 
Therefore, the scaling
distribution exhibit successively more vanishing moments. Given that
the samplets exhibit at least \(q+1\) vanishing moments and assuming
\(\Pcal_q\)-unisolvent data sites, this implies
\[
{\bs T}{\bs P}=\begin{bmatrix}{\bs R}\\ {\bs 0}\end{bmatrix}.
\]
Here, the the upper block \({\bs R}\in\Rbb^{n\times n}\) contains the 
coefficients of the representation of the 
vectors \([p_j({\bs x}_i)]_{i=1,\ldots,N}\) for \(j=1,\ldots, n\)
with respect to the first \(n\) scaling distributions. If 
\(p_1,\ldots,p_n\) is the monomial basis and scaling distributions are sorted
with increasing degree, then the matrix \({\bs R}\) is even upper triangular.

To solve the transformed problem \eqref{trafProb}, it is thus straightforward to 
employ a null space approach, see \cite{BGL05}. A similar reduction
for positive  definite kernels is already found in \cite{CastrillonCandas2024}.
We introduce the orthogonal projector onto the range of \({\bs P}\)
given by
\[
{\bs Q}\isdef {\bs P}({\bs P}^\intercal{\bs P})^{-1}{\bs P}^\intercal.
\]
Then, there holds \({\bs Q}{\bs P}={\bs P}\) as well as \({\bs Q}^{2}={\bs Q}\).
If \({\bs Q}\) is represented in samplet coordinates, we specifically find
\[
{\bs Q}= 
\begin{bmatrix}{\bs R}\\ {\bs 0}\end{bmatrix}({\bs R}^\intercal{\bs R})^{-1}
[{\bs R}^\intercal\ \ {\bs 0}]
=\begin{bmatrix}{\bs I}\\ {\bs 0}\end{bmatrix}[{\bs I}\ \ {\bs 0}]
=\begin{bmatrix}{\bs I} & \\ & {\bs 0}\end{bmatrix}
\]
with \({\bs I}\in\Rbb^{n\times n}\). Similarly, the orthogonal projector
onto the kernel of \({\bs P}\) is
given by
\[
{\bs I}-{\bs Q}=\begin{bmatrix}{\bs 0} & \\ & {\bs I}\end{bmatrix}
\]
with \({\bs I}\in\Rbb^{(N-n)\times(N-n)}\).
Now, applying the null space approach leads to solving
\[
({\bs I}-{\bs Q}){\bs T}{\bs K}_\nu{\bs T}^\intercal({\bs I}-{\bs Q})
{\bs T}{\bs c}=
({\bs I}-{\bs Q}){\bs T}{\bs y}.
\]
To this end, we first observe that 
\({\bs T}{\bs K}_\nu{\bs T}^\intercal={\bs K}^\Sigma+\nu^2{\bs I}\),
since the samplet transform is an isometry.
Next, we partition the samplet transformed kernel matrix into the contribution 
of polynomials and the part orthogonal to polynomials, i.e.,
\[
{\bs T}{\bs K}_\nu{\bs T}^\intercal=\begin{bmatrix}
({\bs K}_{PP}+\nu^2{\bs I}) & {\bs K}_{P\Psi} \\
{\bs K}_{\Psi P} & ({\bs K}_{\Psi\Psi}+\nu^2{\bs I})
\end{bmatrix}
\quad\text{and}\quad
{\bs T}{\bs c}=\begin{bmatrix}{\bs c}_P\\ {\bs c}_{\Psi}
\end{bmatrix},\
{\bs T}{\bs y}=\begin{bmatrix}{\bs y}_P\\ {\bs y}_{\Psi}
\end{bmatrix}.
\]
Consequently, it suffices to solve
\begin{equation}\label{eq:Kpsipsi}
({\bs K}_{\Psi\Psi}+\nu^2{\bs I}){\bs c}_\Psi={\bs y}_\Psi
\end{equation}
and to set \({\bs c}_P={\bs 0}\), which guarantees
\({\bs P}^\intercal{\bs c}={\bs 0}\).
The coefficient vector \({\bs d}\) is then obtained by employing the Moore-Penrose inverse
\[
{\bs P}^\intercal{\bs P}{\bs d}={\bs P}^\intercal
\big({\bs y}-({\bs K}+\nu^2{\bs I}){\bs c}\big).
\]
In samplet coordinates, this becomes
\begin{equation}\label{eq:PolPart}
{\bs R}^\intercal{\bs R}{\bs d}
={\bs R}^\intercal({\bs y}_P-{\bs K}_{P\Psi}{\bs c}_{\Psi})
\quad\Longleftrightarrow\quad
{\bs R}{\bs d}
=({\bs y}_P-{\bs K}_{P\Psi}{\bs c}_{\Psi}),
\end{equation}
since \({\bs c}_P={\bs 0}\) and \({\bs R}\) is invertible.

From the previous derivation, the following theorem is evident.

\begin{theorem}\label{thm:SampSPD}
Let the kernel \(\Kcal\) be 
\(\Pcal_q\)-conditionally positive (semi-)definite and consider samplets with at
least \(q+1\) vanishing moments. Then, the matrix 
\({\bs K}_{\Psi\Psi}=({\bs I}-{\bs Q}){\bs K}^\Sigma({\bs I}-{\bs Q})\) 
is positive (semi-)definite.
\end{theorem}

\section{Numerical results}\label{sec:numerics}
We evaluate the approach under consideration in three settings: 
universal Kriging on the \texttt{Stanford bunny}\footnote{%
\url{http://graphics.stanford.edu/data/3Dscanrep/}\label{fn:stanford}}
and image registration tasks on the two-dimensional 
\texttt{workteam}\footnote{%
\url{https://it.mathworks.com/help/vision/ref/%
vision.cascadeobjectdetector-system-object.html}} image from \textsc{MATLAB} 
as well as on a three-dimensional mesh distortion of the 
\texttt{Stanford armadillo}%
\hyperlink{fn:shared}{\textsuperscript{\ref{fn:stanford}}}.

Throughout all experiments, we employ the samplet-compressed representation of
the kernel matrix described in Theorem~\ref{thm:compression}. The approximation
accuracy is governed by the admissibility parameter, which we set to $\rho = d$
and by the number of vanishing moments. Specifically, we choose $q = 3$ for the
universal Kriging experiment and $q = 5$ for the two image registration
experiments. In addition, we perform an \emph{a posteriori} thresholding of
small matrix entries using a relative tolerance of $\kappa = 10^{-8}$. 

All computations have been carried out on a MacBook Pro equipped with an 
Apple M2 Max processor and 32 GB of main memory. The implementation of the
samplet compression is available as software package
\texttt{FMCA}\footnote{\url{ https://github.com/muchip/fmca}}.

\subsection{Universal Kriging on the Stanford bunny}\label{sec:num_UK}

We illustrate the samplet-based universal Kriging predictor of
Section~\ref{subsec:UniversalKriging} at $N = 300\,000$ scattered data sites
sampled from the Stanford bunny surface with bounding 
box $[0,1]\times[0,0.99]\times[0,0.77]$. As {generalized} covariance we use the 3D
polyharmonic kernel $\Kcal_\ell( r)\isdef -  r/\ell$ with lengthscale parameter
$\ell = 0.5$. 

The ground truth is a deterministic field $f$, consisting of a
polynomial trend $m(\bs x)=\bs p(\bs x)\bs\beta$ and $J=8$ localized features,
i.e.,
\begin{equation}\label{eq:uk_truth_field}
  f(\bs x) = m(\bs x)
    + \sum_{j=1}^{J} a_j\exp\!\big(-\|\bs x-\bs c_j\|_2^2/(2w_j^2)\big),
\end{equation}
with drift coefficients $\bs\beta=(2,1.5,-2,1)^\intercal\in\Rbb^4$. The source
of randomness is the observation noise,
\[
  y_i = f(\bs x_i) + \varepsilon_i,\quad
  \varepsilon_i\stackrel{\text{iid}}{\sim}\Ncal(0,\nu^2),\quad \nu=0.05.
\]
The centers
$\bs c_j$ are shown in Figure~\ref{fig:uk}. Their
coordinates, with the relative amplitudes $a_j$ and widths $w_j$, are reported
in Table~\ref{tab:uk_truth}.

\begin{table}[htb]
\centering
\begin{tabular}{lcccccccc}
\toprule
\(j\)   & 1 & 2 & 3 & 4 & 5 & 6 & 7 & 8\\
\midrule
\(a_j\) & \(0.80\) & \(-0.60\) & \(0.50\) & \(-0.45\)
        & \(0.65\) & \(-0.50\) & \(0.55\) & \(-0.70\)\\
\(w_j\) & \(0.06\) & \(0.10\) & \(0.08\) & \(0.12\)
        & \(0.07\) & \(0.09\) & \(0.11\) & \(0.08\)\\
$x_j$ & 0.11 & 0.49 & 0.62 & 0.12 & 0.14 & 0.04 & 0.19 & 0.90 \\
$y_j$ & 0.78 & 0.02 & 0.62 & 0.29 & 0.24 & 0.50 & 0.78 & 0.36 \\
$z_j$ & 0.42 & 0.28 & 0.56 & 0.62 & 0.60 & 0.47 & 0.07 & 0.46 \\
\bottomrule
\end{tabular}
\caption{Ground-truth field of Section~\ref{sec:num_UK}: coordinates 
  \( \bs c_j = (x_j, y_j, z_j)^\intercal\), amplitudes
\(a_j\) and widths \(w_j\).}
\label{tab:uk_truth}
\end{table}

We remove two subsets of the data before fitting the model. One is a random test
set of $15\,000$ points, where nearby data still exists, to test interpolation.
The other is a \emph{hole}, where we remove all the points inside a ball of
radius $0.12$ around the surface point closest to the anchor
$(1,0.7,0.5)^\intercal$, to test extrapolation. The remaining sites form the
training set on which the predictor is built. 
The posterior mean \eqref{eq:mean-as-interpolant} is a single kernel sum and is
evaluated at all \(N\) sites, while the posterior variance requires by
\eqref{eq:Kriging-variance} one back-substitution against the factor of
\({\bs K}_{\Psi\Psi}+\nu^2{\bs I}\) per evaluation point, and is therefore
computed only at $10\,000$ representative sites, where $2\,000$ are taken from
the hole and $8\,000$ from the test set.

The effect of the observational noise is reflected in the test error, which
is $e_{\mathrm{test}} = 2.44\cdot10^{-2}$ for a noise standard deviation of
$\nu = 0.05$. A more detailed analysis of the error for different noise
levels is provided in Section~\ref{sec:num_error}. Moreover, the posterior
standard deviation at the test sites remains, on average, around
$5\cdot10^{-2}$, while inside the hole it increases by roughly a factor of
$5$, indicating greater uncertainty where data are missing.

Figure~\ref{fig:uk} shows the posterior mean and the posterior standard
deviation on the bunny surface. The mean recovers the trend and the localized
features of the ground truth, while the standard-deviation map is flat over the
sampled region and lights up over the hole.

\begin{figure}[htb]
\centering
\begin{subfigure}{0.32\textwidth}
  \includegraphics[width=\textwidth]{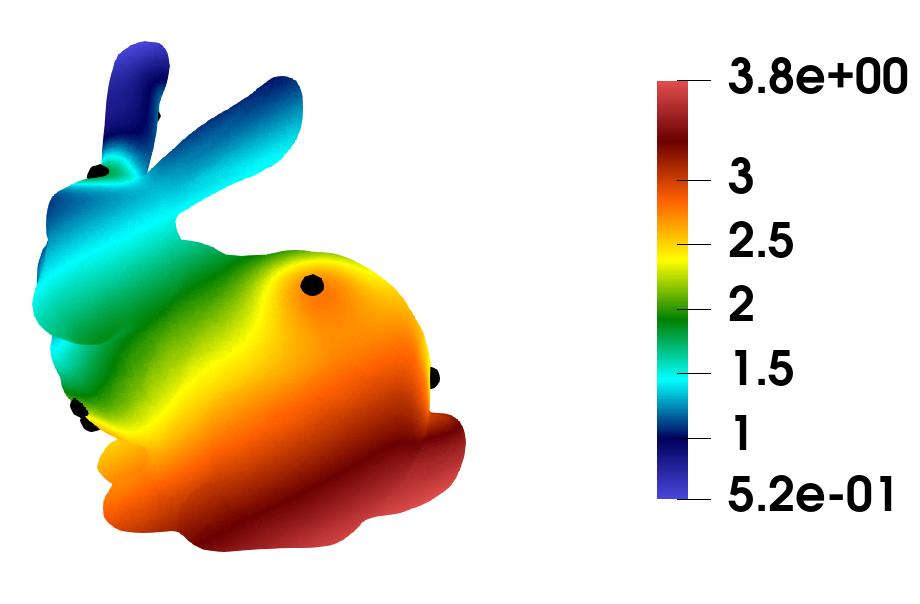}
  \caption{ground truth }
\end{subfigure}\hfill
\begin{subfigure}{0.32\textwidth}
  \includegraphics[width=\textwidth]{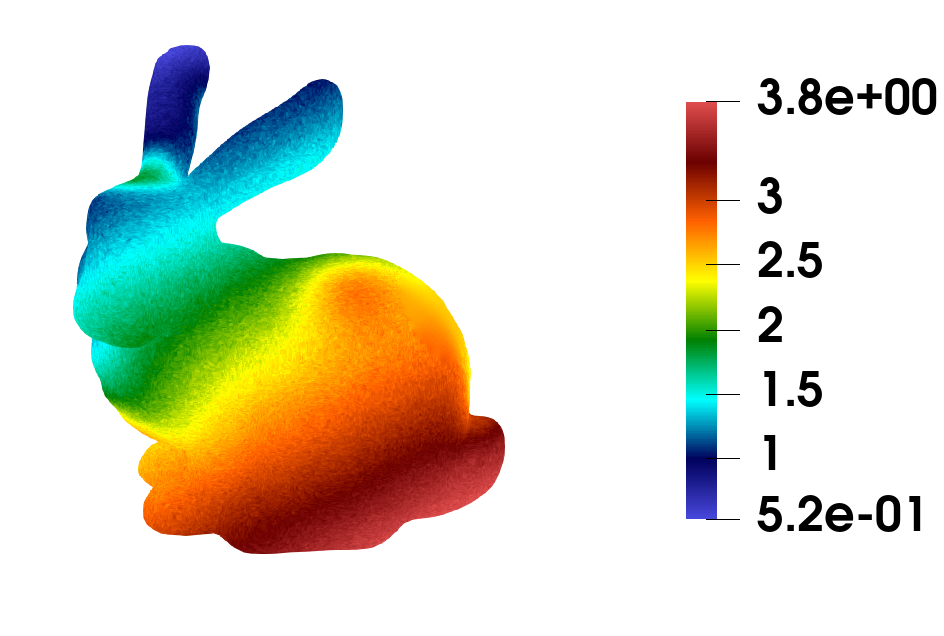}
  \caption{posterior mean }
\end{subfigure}\hfill
\begin{subfigure}{0.32\textwidth}
  \includegraphics[width=\textwidth]{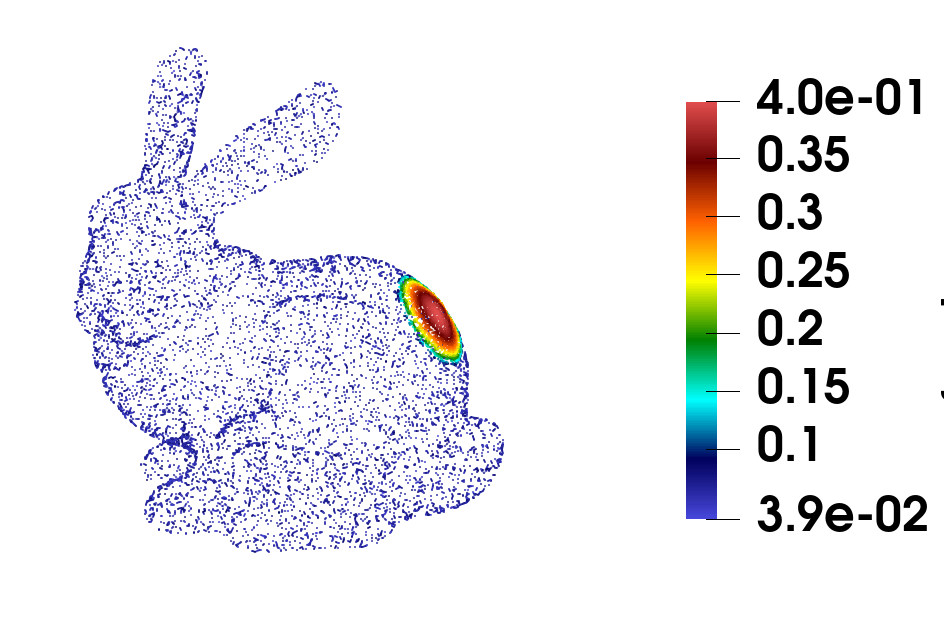}
  \caption{posterior std }
\end{subfigure}
\centering
\begin{subfigure}{0.32\textwidth}
  \includegraphics[width=\textwidth]{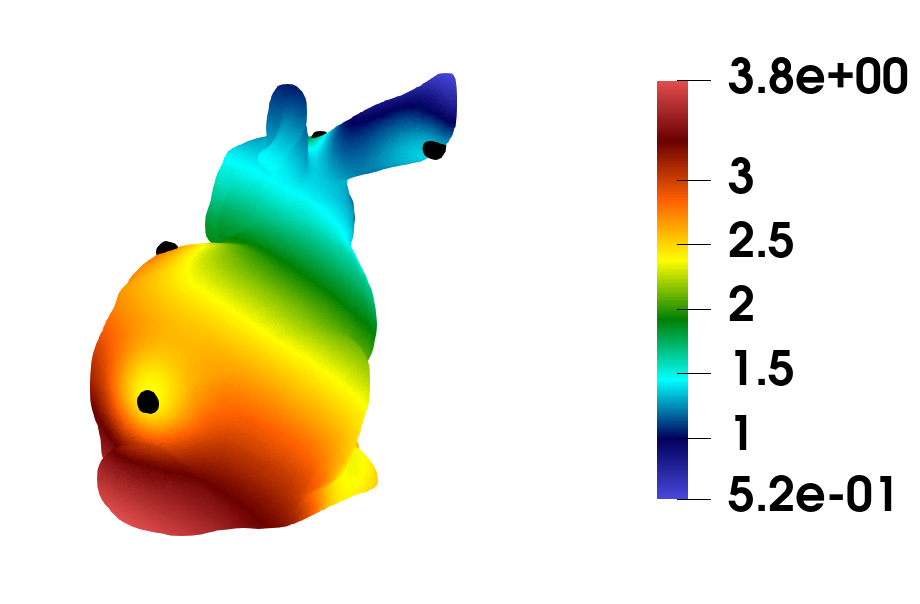}
  \caption{ground truth }
\end{subfigure}\hfill
\begin{subfigure}{0.32\textwidth}
  \includegraphics[width=\textwidth]{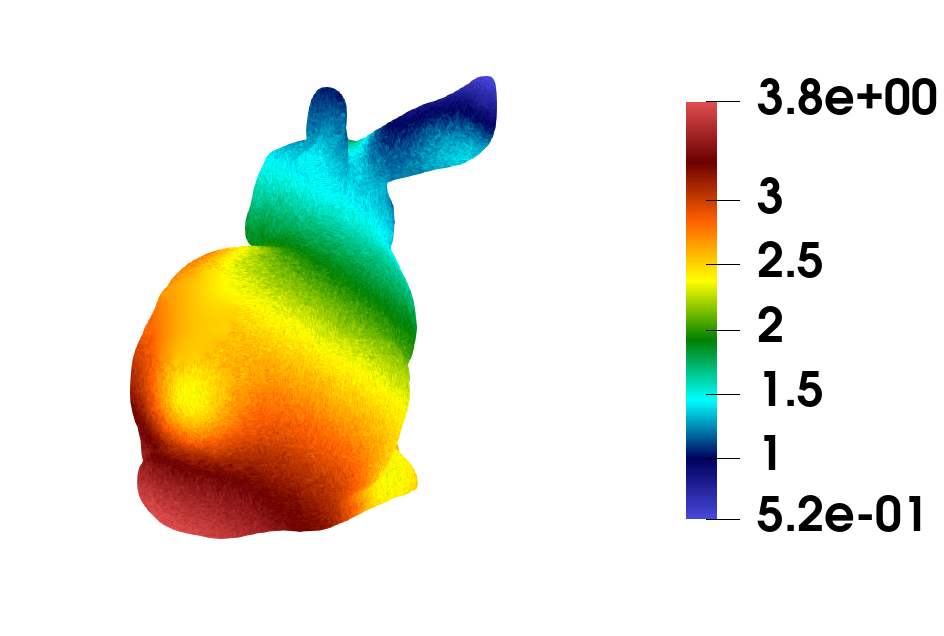}
  \caption{posterior mean }
\end{subfigure}\hfill
\begin{subfigure}{0.32\textwidth}
  \includegraphics[width=\textwidth]{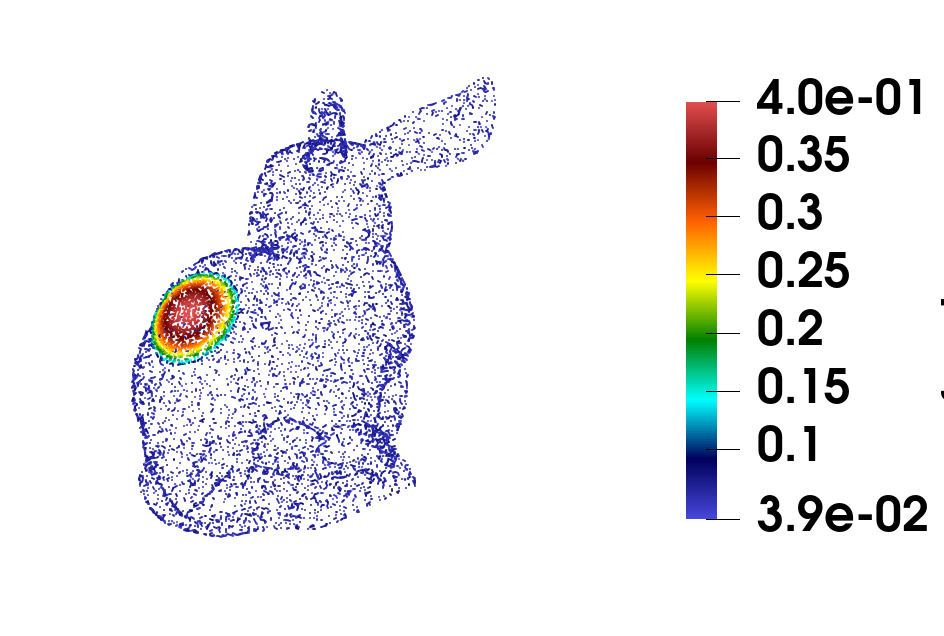}
  \caption{posterior std }
\end{subfigure}
\caption{Universal Kriging at the bunny surface. The posterior mean (B-E)
reproduces the reference field (A-D) up to the noise level, and the posterior
standard deviation (C-F) is small over the sampled surface and increases inside
the withheld data hole. The black dots appearing on Figures (A-D) are the centers
of the 8 bumps.}
\label{fig:uk}
\end{figure}

\subsection{Two-dimensional image warping}\label{sec:num_registration}
As further application, we consider landmark-based image warping for geometric
distortion correction. Let $ D\subseteq\Rbb^2$ be a domain. Given $N$ landmark
pairs $\{(\bs x_i,\bs y_i)\}_{i=1}^N\subseteq D\times D$, we fit a deformation
field $\bs\Phi\colon D\to D$ with $\bs\Phi(\bs x_i)=\bs y_i$ that is as smooth
as possible between the landmarks. Each component $\Phi_k$, $k=1,2$,  is
approximated using the thin plate spline kernel
\(
  \Kcal( r)\isdef r ^2\log (r),
\)
which is $\Pcal_1$-conditionally positive definite in $\Rbb^2$. Each component
of the deformation field minimizes the
bending energy
\begin{equation}\label{eq:bending}
  J[\Phi_k]=\int_{\Rbb^2}\big[(\partial_{11}^2\Phi_k)^2
    +2(\partial_{12}^2\Phi_k)^2+(\partial_{22}^2\Phi_k)^2\big] d\bs x
\end{equation}
subject to $\Phi_k(\bs x_i)=y_{i,k}$,
see \cite{Duchon1977,bookstein1997landmark}.
Moreover, there holds the representation
\begin{equation}\label{eq:tps-rep}
  \Phi_k(\bs x)=\sum_{j=1}^N c_{j,k} \Kcal(\bs x,\bs x_j)
    +\sum_{\ell=1}^{3}d_{\ell,k} p_\ell(\bs x),
  \qquad \bs P^\intercal\bs c_k=\bs 0,
\end{equation}
where $\bs c_k,\bs d_k$ solve the CPD saddle-point system
\eqref{eq:saddlesystem_noiseless} with right-hand side
$\bs y_k=[y_{1,k},\dots,y_{N,k}]^\intercal$. 

Both components share the centers $\bs x_1,\ldots, \bs x_N$. Hence, a single
samplet compression of $\bs K$ and a single Cholesky factorization of the
regularized samplet-block $\bs K_{\Psi\Psi}+\lambda \bs I$
are computed once and reused for the two right-hand sides $\bs y_1,\bs y_2$ through the null-space solver of Section~\ref{sec:SampletsCPD}.

To validate the method, we deform the test image with a known map
$\bs g\colon [0,1]^2\to[0,1]^2$ that rotates points about the centre
$\bs\gamma=(\tfrac12,\tfrac12)$ by an angle that is largest at the centre and
fades to zero at radius $R = 0.4$. The deformation also includes two localized
Gaussian displacements. With $r_\gamma=\|\bs x-\bs\gamma\|_2$ and rotation angle
\[
  \theta(r_\gamma)=(1-r_\gamma^2/R^2)^2\ \ \text{for } r_\gamma<R,
  \qquad \theta(r_\gamma)=0\ \ \text{otherwise},
\]
the deformation field hence reads
\begin{equation}\label{eq:gtruth}
  \bs g(\bs x)=\bs\gamma
   +\begin{bmatrix}\cos(\theta(r_\gamma)) & -\sin(\theta(r_\gamma))\\[2pt]
                   \sin(\theta(r_\gamma)) & 
                 \phantom{-}\cos(\theta(r_\gamma))\end{bmatrix}
    (\bs x-\bs\gamma)
   +\sum_{m=1}^{2}\bs a_m \exp\!\big(-\|\bs x-\bs b_m\|^2/(2w_m^2)\big).
\end{equation}
Since $\theta(R)=0$, the map is the identity outside the disk of radius 0.4.
In particular, $\bs g$ maps $[0,1]^2$ into itself and is a diffeomorphism for 
small $\theta_0$. The parameters are listed in Table~\ref{tab:warp_gt}.

\begin{table}[htb]
\centering
\begin{tabular}{ccccc}
\toprule
 & $\bs b_m$ & $\bs a_m$ & $w_m$\\
\midrule
swirl   & $(0.5,0.5)$ & --- & ---\\
bump 1 & $(0.30,0.30)$ & $(0.060, 0.040)$ & $0.12$\\
bump 2  & $(0.70,0.65)$ & $(-0.050, 0.050)$ & $0.10$\\
\bottomrule
\end{tabular}
\caption{Parameters of the ground-truth deformation \eqref{eq:gtruth} on the unit
square.}
\label{tab:warp_gt}
\end{table}

The source landmarks $\{\bs x_i\}_{i=1}^N$ form a regular $256\times256$ grid on
$[0.05,0.95]^2$, so that $N=65\,536$, with targets $\bs y_i=\bs g(\bs x_i)$.
We fit both the \emph{forward} warp $\bs\Phi\approx\bs g$, which quantifies the
reconstruction accuracy, and the \emph{inverse} warp $\bs\Psi\approx\bs g^{-1}$,
which renders the deformed image, see \cite{wolberg1990,glasbey1998review} for
details. We remark that forward warping, which pushes each input
pixel through $\bs\Phi$, may map onto non-integer locations and thus leave
unfilled holes in the output. 
Backward warping avoids this, since for every \emph{output} pixel
$\bs p$ one reads the input intensity at its pre-image,
\begin{equation}\label{eq:backwarp}
  I_{\mathrm{warp}}(\bs p)=I\big(\bs\Psi(\bs p)\big),
\end{equation}
so that every output pixel receives a value. Since the pre-image 
$\bs\Psi(\bs p)$ generally falls between input pixels, the intensity 
$I(\bs\Psi(\bs p))$ is obtained by resampling, here nearest-neighbor
interpolation, following the classical approach as in
\cite[Chapter 3.1]{wolberg1990}.

We assess the reconstruction with three quantities. The landmark residual
\[
  e_{\mathrm{land}}\isdef\bigg(\frac1N\sum_{i=1}^N
    \|\bs\Phi(\bs x_i)-\bs y_i\|^2\bigg)^{1/2}
\]
is the root-mean-square error at the landmarks. The remaining two are evaluated
on a regular $1000\times 1000$ grid $G\subset[0,1]^2$ and are the relative
$\ell_2$ field-recovery error
\[
e_{\bs\Phi}\isdef\|\bs\Phi-\bs g\|_{\ell_2(G)}/\|\bs g\|_{\ell_2(G)},
\]
and
the minimum Jacobian determinant 
\[
J_{\min}\isdef\min_{G}\det\nabla\bs\Phi
\]
computed by central differences. A positive $J_{\min}$ indicates an
orientation-preserving, fold-free deformation, which is the standard
admissibility criterion for a warp, see  \cite{sotiras2013survey}. 
The discretization parameters and the quantities of interests are
listed in Table~\ref{tab:warp_param} and the results are shown in
Figures~\ref{fig:warp_field} and~\ref{fig:warp_image}.
The regularization parameter is chosen rather small, i.e., $\lambda = 10^{-6}$,
and only serves preserving positive definiteness of the samplet compressed
and thresholded kernel matrix, so that the Cholesky
factorization is applicable. Its effect on the fit is negligible,
as the landmark residual $e_{\mathrm{land}}$ shows.

\begin{table}[htb]
\centering
\begin{tabular}{ccccccc}
\toprule
  $N$ & $N_{\text{eval}}$ & $\lambda$ & $e_{\mathrm{land}}$ & $e_{\bs\Phi}$ &
  $J_{\min}$ & $t_{\text{fit}} ~ [s]$\\
\midrule
$65\,536$ & $1\,048\,576$ & $10^{-6}$ & $1.3\cdot10^{-7}$ & $1.5\cdot10^{-4}$ &
0.37 & 44.85 \\
\bottomrule
\end{tabular}
\caption{Warping experiment, parameters and results. From left to right, we have
  the number of landmarks $N$, the cardinality of the
evaluation grid $N_{\text{eval}}$, the ridge parameter $\lambda$,
followed by the RMS landmark residual
$e_{\mathrm{land}}$, the relative $\ell_2$ field-recovery error $e_{\bs\Phi}$,
the minimum Jacobian determinant $J_{\min}$, and the fitting time $t$
for both warps.}
\label{tab:warp_param}
\end{table}

\begin{figure}[htb]
\tikzexternaldisable 
\centering
\begin{subfigure}{0.32\textwidth}\centering
  \begin{tikzpicture}\begin{axis}[warpaxis]
    \addplot graphics [xmin=0,xmax=1,ymin=0,ymax=1] {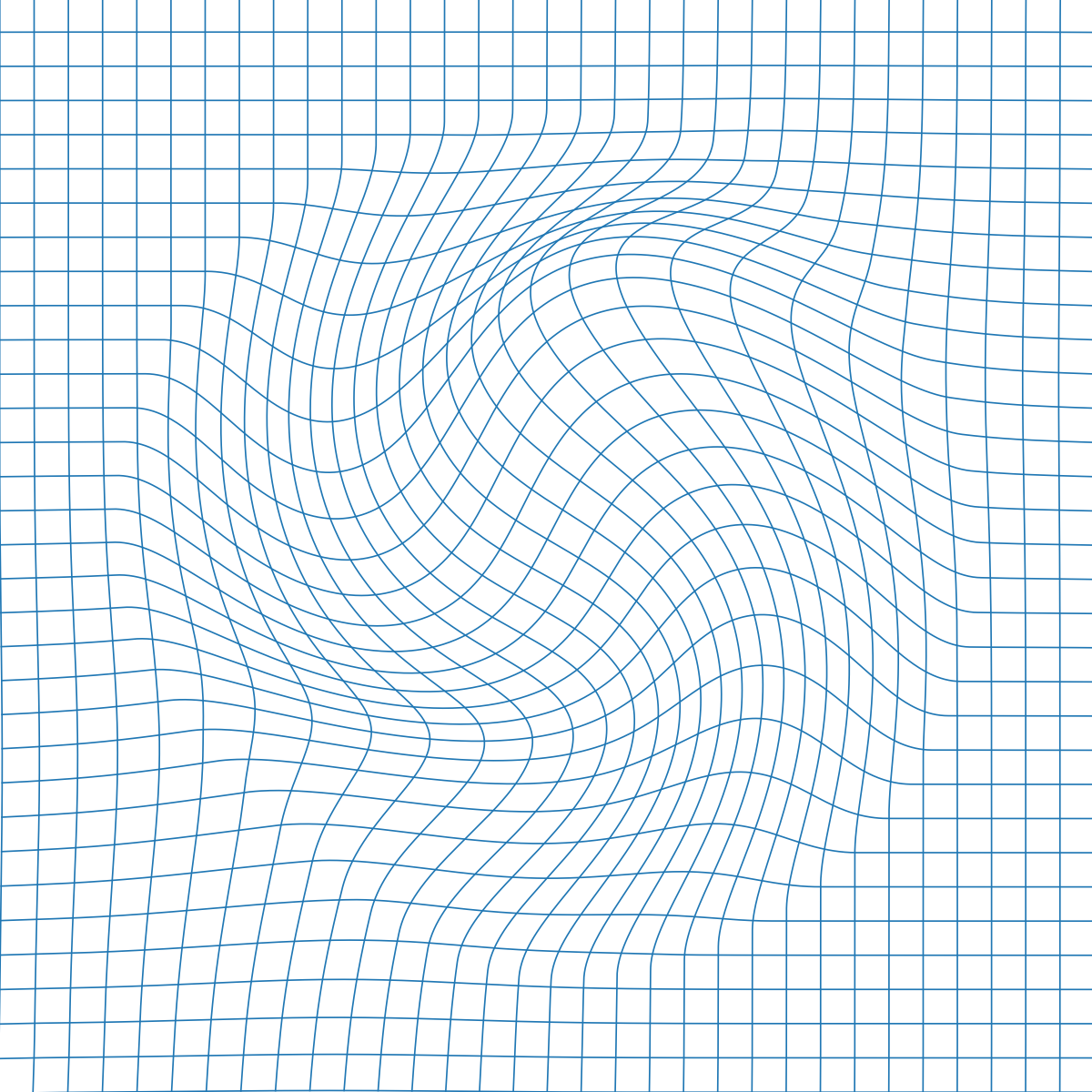};
  \end{axis}\end{tikzpicture}
  \caption{deformed grid $\bs\Phi$}
\end{subfigure}
\begin{subfigure}{0.32\textwidth}\centering
  \begin{tikzpicture}\begin{axis}[warpaxis, warpcbar, colormap name=viridis,
        point meta min=0, point meta max=0.174,
        colorbar style={ytick={0,0.05,0.1,0.15}}]
    \addplot graphics [xmin=0,xmax=1,ymin=0,ymax=1] { 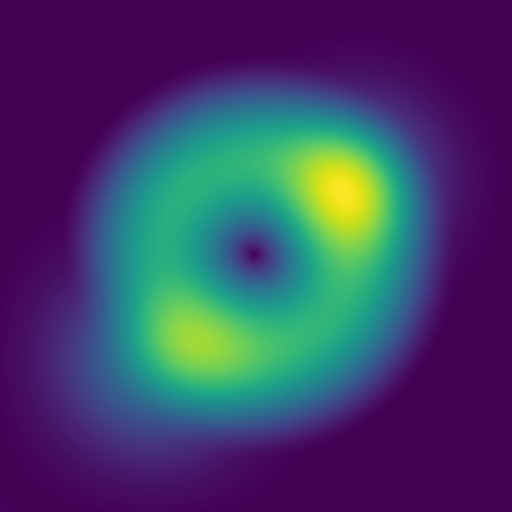};
  \end{axis}\end{tikzpicture}
  \caption{displacement $\|\bs\Phi(\bs x)-\bs x\|_2$}
\end{subfigure}
\begin{subfigure}{0.32\textwidth}\centering
\begin{tikzpicture}\begin{axis}[warpaxis, warpcbar, colormap name=coolwarm,
      point meta min=-1.021, point meta max=1.021,
      colorbar style={
        ytick={-1.02, -0.5, 0, 0.5, 1.02},        
        yticklabels={-1.02, -0.5, 0, 0.5, 1.02}}]
  \addplot graphics [xmin=0,xmax=1,ymin=0,ymax=1] {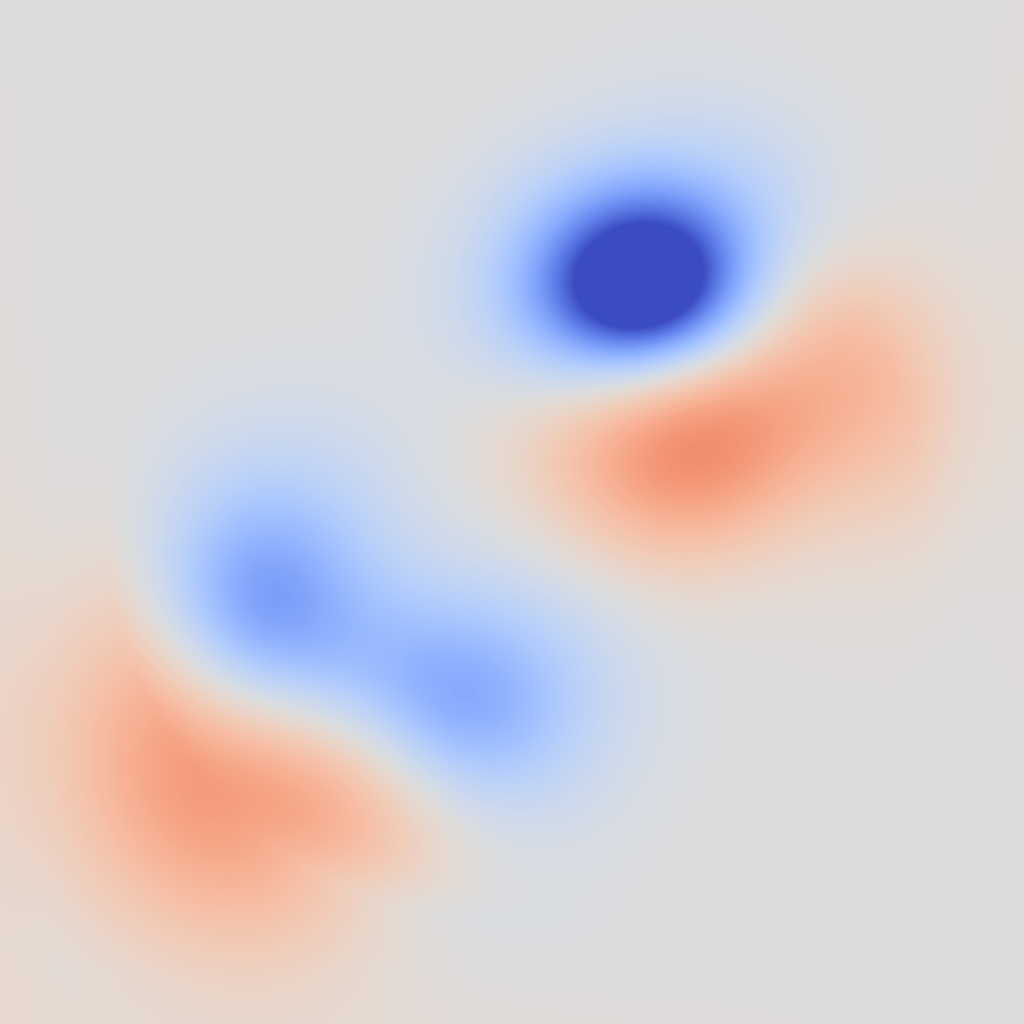};
\end{axis}\end{tikzpicture}
  \caption{$\log(\det\nabla\bs\Phi)$}
\end{subfigure}
\caption{Recovered deformation $\bs\Phi$ from $N=65\,536$ landmarks.}
\label{fig:warp_field}
\end{figure}

\begin{figure}[htb]
\tikzexternaldisable 
\centering
\begin{subfigure}{0.32\textwidth}\centering
  \begin{tikzpicture}\begin{axis}[warpaxis]
    \addplot graphics [xmin=0,xmax=1,ymin=0,ymax=1] { 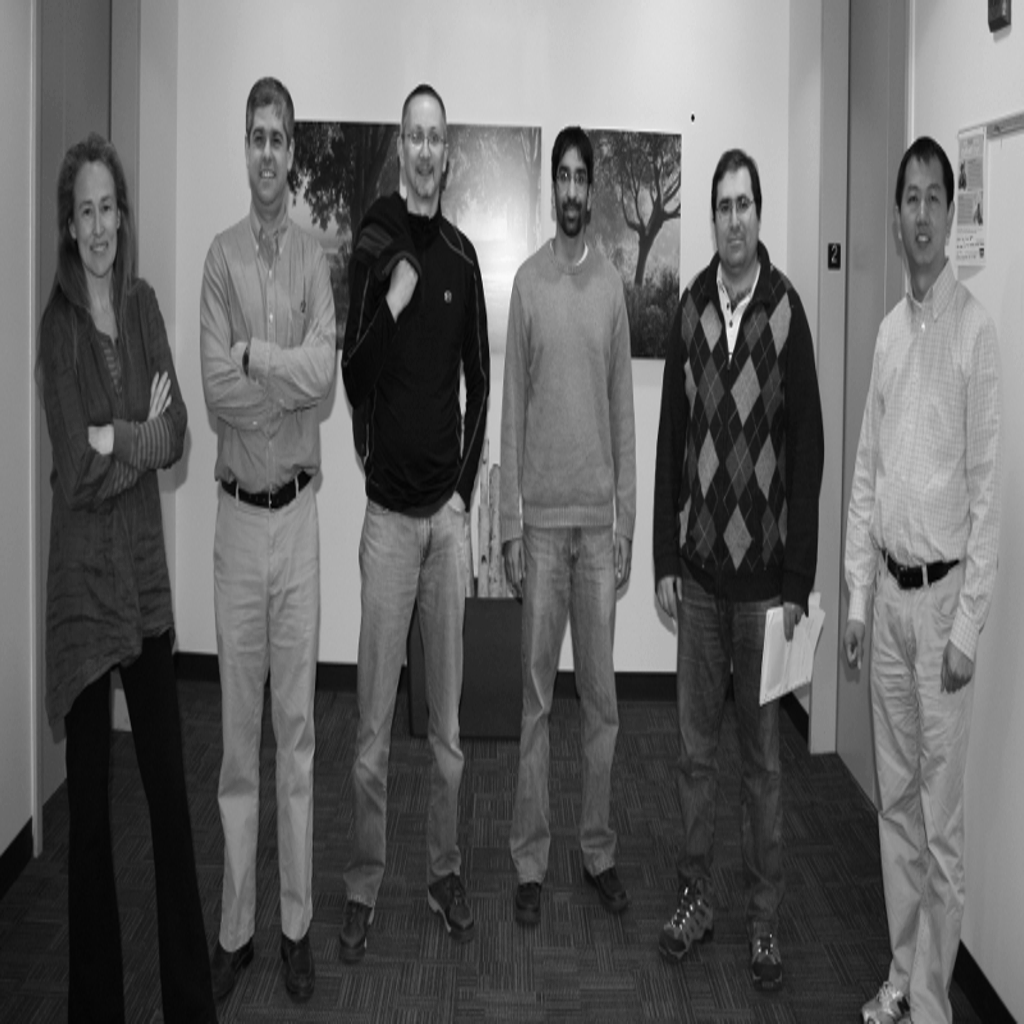};
  \end{axis}\end{tikzpicture}
  \caption{original}
\end{subfigure}
\begin{subfigure}{0.32\textwidth}\centering
  \begin{tikzpicture}\begin{axis}[warpaxis]
    \hspace{-0.5 cm} \addplot graphics [xmin=0,xmax=1,ymin=0,ymax=1] { 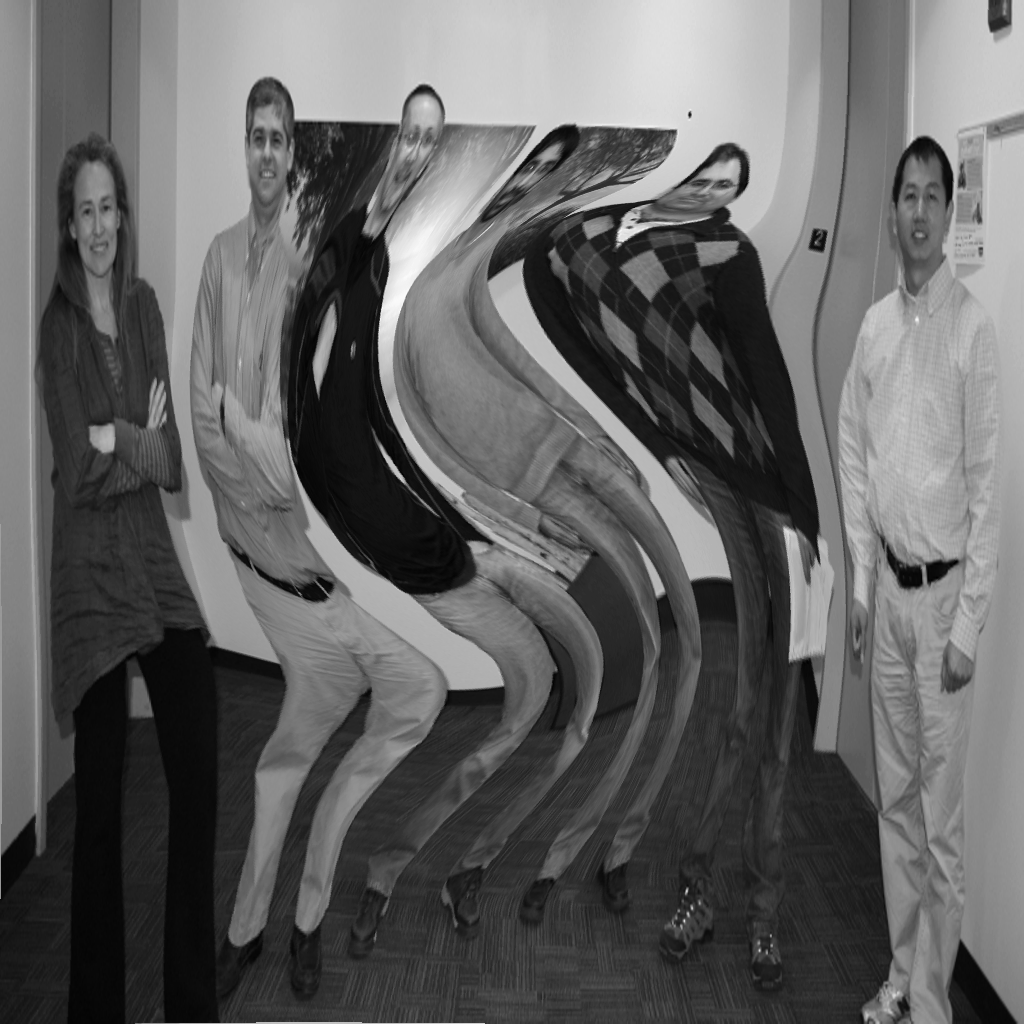};
  \end{axis}\end{tikzpicture}
  \hspace{-0.5 cm}  \caption{warped}
\end{subfigure}
\begin{subfigure}{0.32\textwidth}\centering
  \begin{tikzpicture}\begin{axis}[warpaxis, warpcbar, colormap name=coolwarm,
        point meta min=-224, point meta max=224,
        colorbar style={ytick={-200,-100,0,100,200}}, ]
    \addplot graphics [xmin=0,xmax=1,ymin=0,ymax=1] { 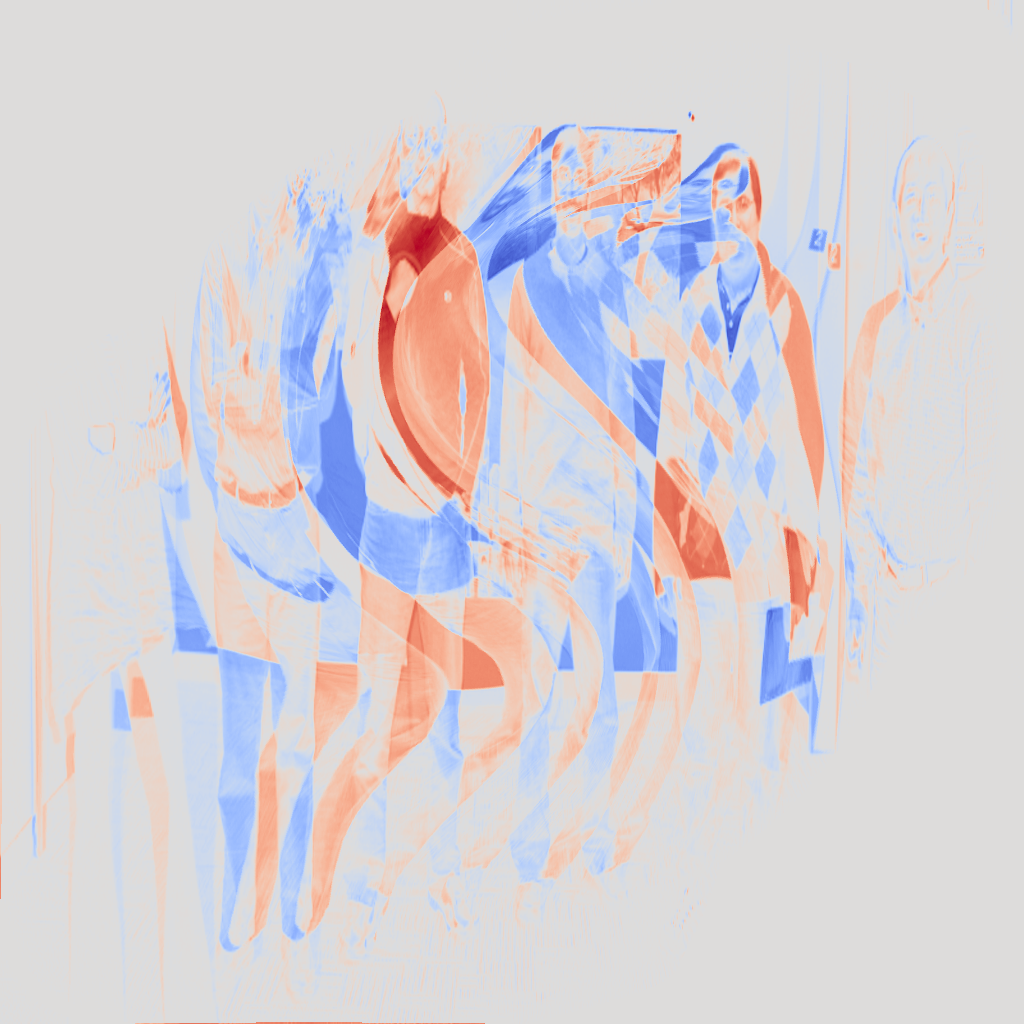};
  \end{axis}\end{tikzpicture}
  \caption{warped $-$ original}
\end{subfigure}
\caption{Backward warping of the test image through the inverse map $\bs\Psi$:
original (A), warped result (B), and their difference (C).}
\label{fig:warp_image}
\end{figure}


\subsection{Three-dimensional mesh deformation}\label{sec:num_armadillo}

As a three-dimensional counterpart of the planar image-warping test, we consider
landmark-based deformation of the Stanford armadillo. The reference geometry is
given by a high-resolution mesh with \(N_{\mathrm{all}}=2\,798\,914\)
vertices, from which we extract a subset of \(N=30\,096\) landmarks.
The geometry is contained in the bounding box
$
[-0.44,0.40]
\times [-0.57,0.43]
\times [-0.38,0.38].
$
We prescribe a synthetic ground-truth diffeomorphism
\(\bs g\colon\mathbb R^3\to\mathbb R^3\) acting on the coordinates
\(\bs p=[x,y,z]^\intercal\).
The deformation combines a twist around the vertical
axis, a parabolic bend, a radial stretch, and two localized Gaussian
displacements. With
\(
  \theta(y)=1.2y,
\)
the deformation is defined by
\begin{equation}\label{eq:gtruth_armadillo}
\bs g(\bs p)
={}
\begin{bmatrix}
\cos(\theta(y)) & 0 & -\sin(\theta(y))\\
0 & 1 & 0\\
\sin(\theta(y)) & 0 & \cos(\theta(y))
\end{bmatrix}
\bs p
+
\begin{bmatrix}
0.35y^2\\
0\\
0
\end{bmatrix}
+
0.6\|\bs p\|_2^2\bs p +
\sum_{m=1}^{2}\bs a_m
\exp\!\left(
-\frac{\|\bs p-\bs b_m\|_2^2}{2w_m^2}
\right).
\end{equation}
The parameters of the two localized Gaussian perturbations are listed in
Table~\ref{tab:warp_gt_armadillo}.

\begin{table}[htb]
\centering
\begin{tabular}{lccc}
\toprule
 & \(\bs b_m\) & \(\bs a_m\) & \(w_m\)\\
\midrule
bump 1 & \((0.12, 0.20, 0.10)\) & \((0.05, 0.00, 0.03)\) & \(0.12\) \\
bump 2 & \((-0.10, -0.18, 0.08)\) & \((-0.04, 0.03, 0.00)\) & \(0.10\) \\
\bottomrule
\end{tabular}
\caption{Parameters of the two localized Gaussian perturbations in the
ground-truth deformation \eqref{eq:gtruth_armadillo}.}
\label{tab:warp_gt_armadillo}
\end{table}

We then fit the forward deformation
\(\bs\Phi=[\Phi_1,\Phi_2,\Phi_3]^\intercal\approx \bs g\) from the landmark 
pairs \((\bs x_i,\bs y_i)\), where \(\bs y_i=\bs g(\bs x_i)\). Each component
\(\Phi_k\), \(k=1,2,3\), is represented with the 3D biharmonic kernel
\(
\mathcal K(r)=-r.
\)
Again, since all three components share the same landmark set, we compress the
kernel matrix only once in the samplet basis and reuse one Cholesky
factorization of the regularized block \(\bs K_{\Psi\Psi}+\lambda \bs I\), with
\(\lambda=10^{-6}\), for the three right-hand sides. The resulting kernel
representation is then evaluated on the full armadillo mesh, using the fast
multipole method.

The results are collected in Table~\ref{tab:armadillo_param}, while
Figure~\ref{fig:armadillo} shows the original armadillo, the deformed geometry,
and a color map of the pointwise field-recovery error.

\begin{table}[htb]
\centering
\begin{tabular}{ccccccc}
\toprule
  $N$ & $N_{\text{eval}}$ & $\lambda$ & $e_{\mathrm{land}}$ & $e_{\bs\Phi}$ & $t_{\mathrm{FMM}}~[s]$ & $t_{\mathrm{fit}}~[s]$\\
\midrule
$30\,096$ & $2\,798\,914$ & $10^{-6}$ & $6.9\cdot10^{-9}$ & $6.6\cdot10^{-5}$ & $ 902.36 $ & $40.47$ \\
\bottomrule
\end{tabular}
\caption{Three-dimensional armadillo deformation: number of landmarks $N$, number of evaluation points $N_{\text{eval}}$, ridge parameter \(\lambda\), RMS landmark residual \(e_{\mathrm{land}}\), relative \(\ell_2\) field-recovery error \(e_{\bs\Phi}\), fast multipole evaluation time \(t_{\mathrm{FMM}}\), and fitting time \(t_{\mathrm{fit}}\).}
\label{tab:armadillo_param}
\end{table}

\begin{figure}[htb]
\centering

\begin{subfigure}{0.32\textwidth}\centering
  \includegraphics[width=\textwidth]{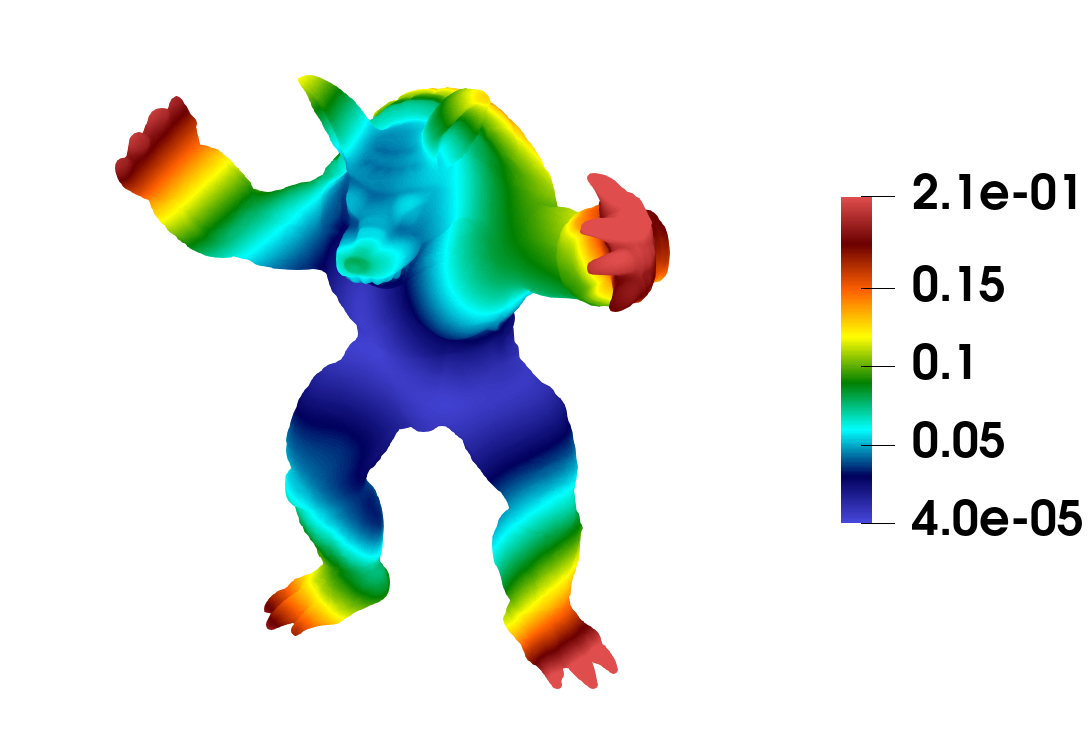}
  \caption{original}
\end{subfigure}\hfill
\begin{subfigure}{0.32\textwidth}\centering
  \includegraphics[width=\textwidth]{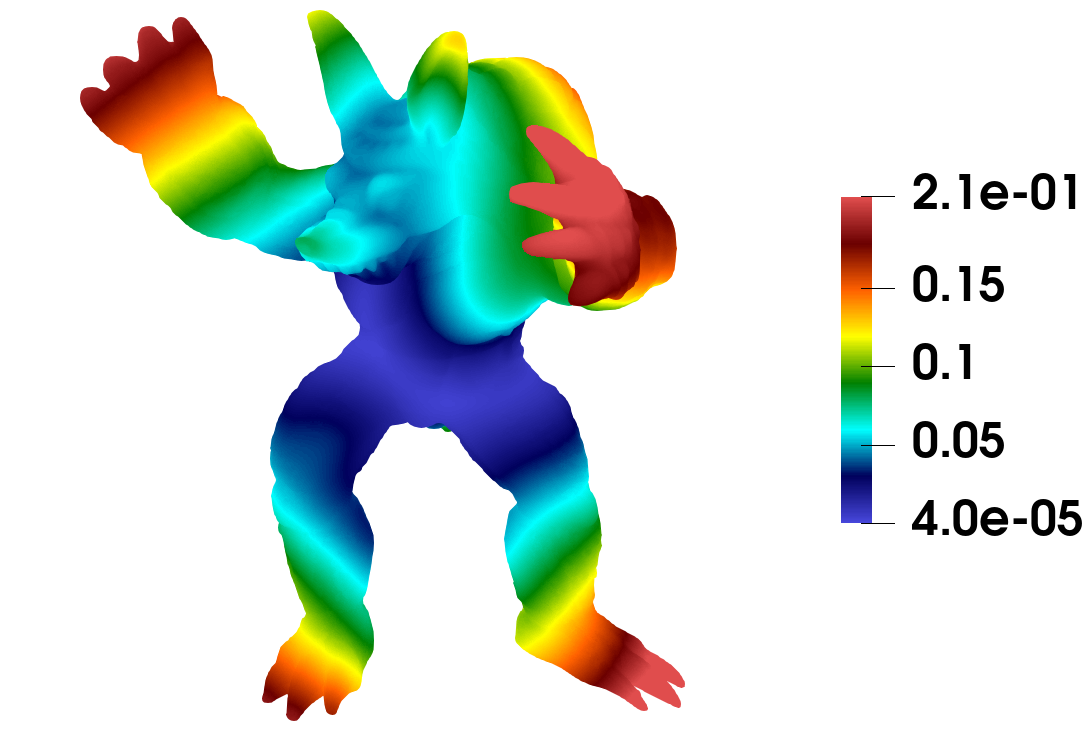}
  \caption{deformed}
\end{subfigure}\hfill
\begin{subfigure}{0.32\textwidth}\centering
  \includegraphics[width=\textwidth]{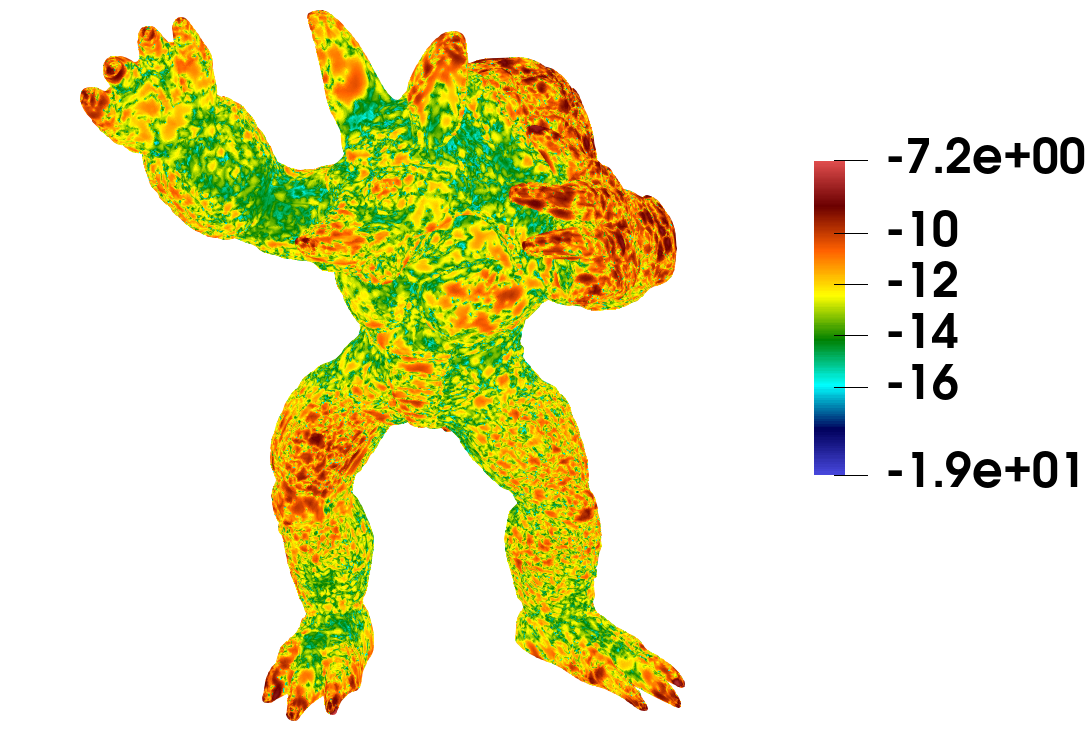}
  \caption{log. error}
\end{subfigure}

\vspace{0.8em}

\begin{subfigure}{0.32\textwidth}\centering
  \includegraphics[width=\textwidth]{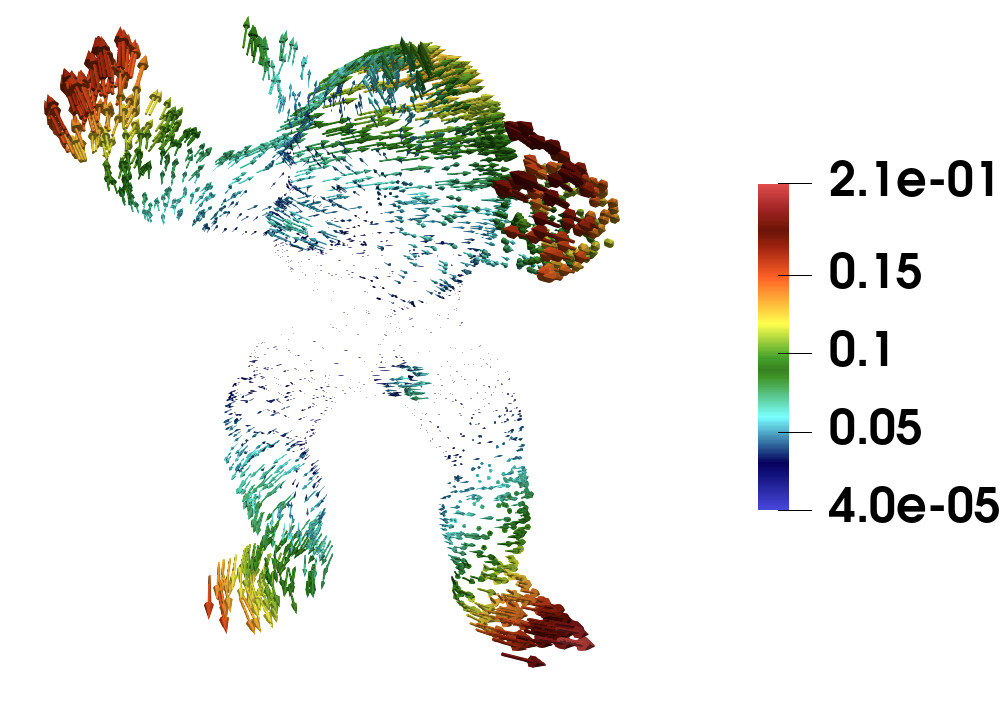}
  \caption{deformation field}
\end{subfigure}
\quad \quad 
\begin{subfigure}{0.32\textwidth}\centering
   \includegraphics[width=\textwidth]{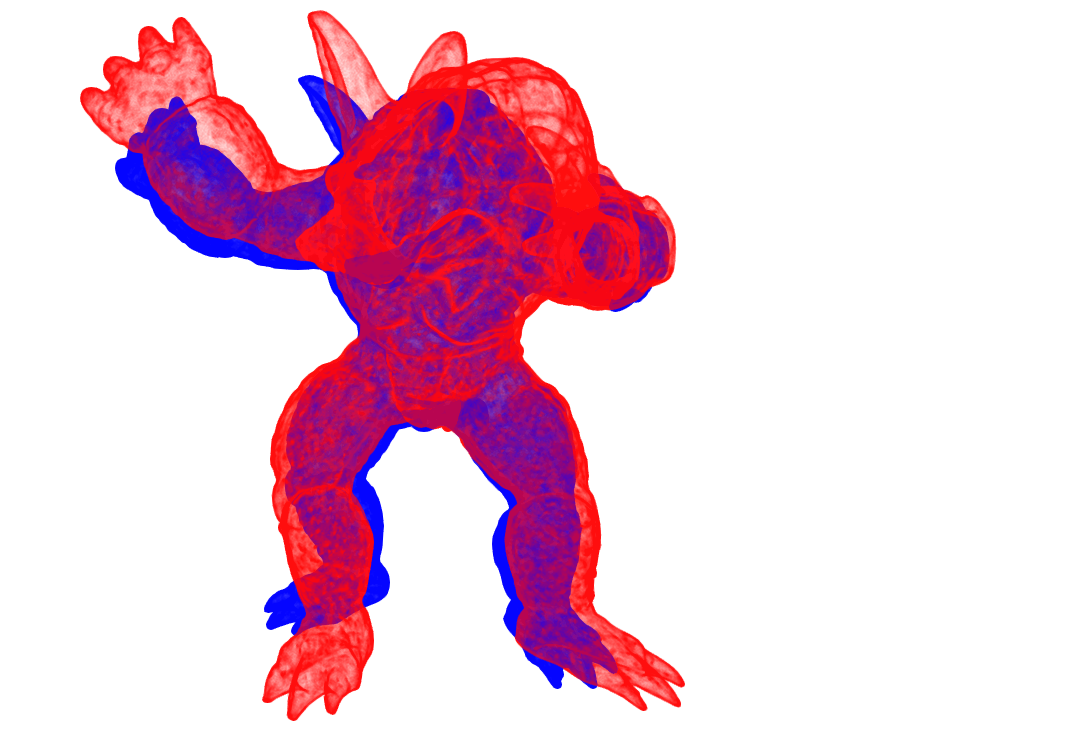}
  \caption{warped geometry}
\end{subfigure}

\caption{Three-dimensional armadillo deformation: original geometry, 
deformed geometry, pointwise logarithmic error, deformation field
and warped shape.}
\label{fig:armadillo}
\end{figure}

\subsection{Validation of the error estimates}\label{sec:num_error}

We conclude by verifying the error estimate of Theorem~\ref{thrm:Error} and its
corollary in the noisy universal Kriging experiment. We measure the error in the
discrete
$\ell_2$-norm, which corresponds to $s=0$ and $r=2$, so that $\gamma=2$ and
$\big(\tfrac12-\tfrac1r\big)_+=0$, and the corollary reduces to
\begin{equation}\label{eq:error_reduced}
  \mathbb E\big(\|Z-\mu\|_{L_2( D)}\big)\ \leq\
  C\big(h_{X, D}^{k}+\nu h_{X, D}^{d/2}\big)|Z|_{B\!L_k( D)}
  + C\rho_{X, D}^{d/2}\big(h_{X, D}^{k-d/2}+\nu\big).
\end{equation}
The first three terms decay under refinement, whereas the last term,
$C\rho_{X, D}^{d/2}\nu$, is independent of $h_{X, D}$ in case of quasi-uniform
data sites. In particular,
it is an irreducible floor proportional to the noise $\nu$ that refinement
does not improve. The numerical convergence analysis confirms the two regimes.

For the Kriging experiment of Section~\ref{sec:num_UK}, we refine the training
set by subsampling the point cloud and measure the error of the posterior mean
on a fixed test set of $15\,000$ sites, for
$\nu\in\{0,0.0001,0.0005,0.001,0.005,0.01,0.05\}$. The polyharmonic kernel
has $q=1$, hence $k=2$. The resulting convergence curves are collected in
Figure~\ref{fig:uk_convergence} and exhibit the two regimes predicted by
\eqref{eq:error_reduced}. Without noise, the error decays quadratically,
matching $h_{X, D}^{2}$, while, in the presence of noise, each curve follows this
rate as long as the approximation error exceeds
$C\rho_{X, D}^{d/2}\nu$, and saturates once the two are comparable. For a
given noise level, this transition identifies the resolution beyond which
further refinement of the data no longer improves the prediction. Note that the
mesh ratio $\rho_{X, D}$ enters only through the constant. Indeed, the
nested subsamples of the point cloud are quasi-uniform, so $\rho_{X, D}$ is
uniformly bounded during the refinement and $C\rho_{X, D}^{d/2}$ is a fixed
constant, whose value is determined by the experiment: for the noise level
$\nu=0.05$, used in the experiment, the error saturates at approximately
$2.4\cdot10^{-2}$, that is, $C\rho_{X, D}^{d/2}\approx 1/2$.

We remark that in all the experiments the cost of $\Ocal(N\log N)$ of the
samplet compression renders the
refinement, up to $N\approx300\,000$ data sites, computationally feasible
without affecting the convergence behavior.

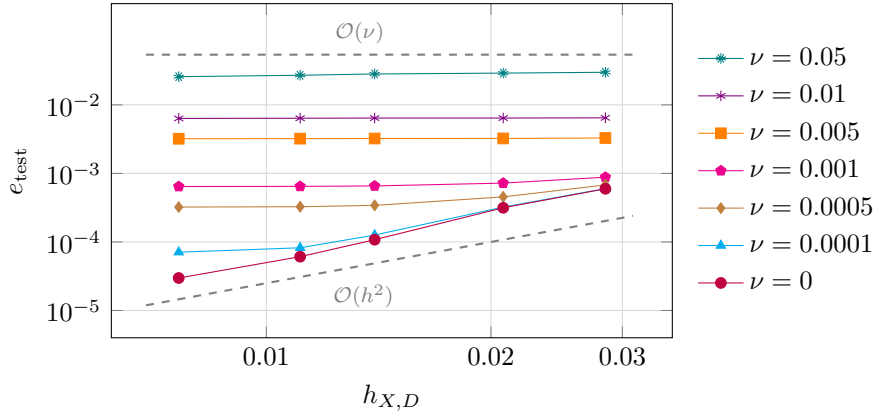
\begin{figure}[htb]
\tikzexternaldisable
\centering
\begin{tikzpicture}
\begin{loglogaxis}[
  width=9cm, height=6cm,
  xlabel={$h_{X, D}$}, ylabel={$e_{\mathrm{test}}$},
  legend style={at={(1.03,0.5)}, anchor=west, draw=none, fill=none,
                font=\small, row sep=2pt},
  legend cell align=left,
  grid=both, major grid style={gray!30}, minor grid style={gray!12},
  tick label style={font=\small}, label style={font=\small},
  log basis x=10, log basis y=10,
  xmin=0.0062, xmax=0.035, ymin=4e-6, ymax=3e-1,
  ytick={1e-5,1e-4,1e-3,1e-2},
  xtick={0.01,0.02,0.03}, xticklabels={0.01,0.02,0.03},
]
\addplot[gray, dashed, thick, forget plot]
  coordinates {(0.0069,0.054) (0.031,0.054)};
\node[gray, font=\scriptsize, anchor=north west]
  at (axis cs:0.012,0.23) {$\mathcal{O}(\nu)$};
\addplot[gray, dashed, thick, forget plot, domain=0.0069:0.031, samples=2]
  {0.25*x^2};
\node[gray, font=\scriptsize, anchor=north]
  at (axis cs:0.0135,3.4e-5) {$\mathcal{O}(h^{2})$};
\addplot[teal, mark=10-pointed star] coordinates {
  (0.0284671, 2.98963e-2)
  (0.0207682, 2.90765e-2)
  (0.0139686, 2.82399e-2)
  (0.0110998, 2.70565e-2)
  (0.00763247,2.57792e-2)
};
\addlegendentry{$\nu=0.05$}

\addplot[violet, mark=asterisk] coordinates {
  (0.0284671, 6.47004e-3)
  (0.0207682, 6.42426e-3)
  (0.0139686, 6.42448e-3)
  (0.0110998, 6.38369e-3)
  (0.00763247,6.34279e-3)
};
\addlegendentry{$\nu=0.01$}

\addplot[orange, mark=square*] coordinates {
  (0.0284671, 3.28563e-3)
  (0.0207682, 3.23899e-3)
  (0.0139686, 3.23191e-3)
  (0.0110998, 3.21673e-3)
  (0.00763247,3.20464e-3)
};
\addlegendentry{$\nu=0.005$}

\addplot[magenta, mark=pentagon*] coordinates {
  (0.0284671, 8.80404e-4)
  (0.0207682, 7.21976e-4)
  (0.0139686, 6.56685e-4)
  (0.0110998, 6.46867e-4)
  (0.00763247,6.43798e-4)
};
\addlegendentry{$\nu=0.001$}

\addplot[brown, mark=diamond*] coordinates {
  (0.0284671, 6.78912e-4)
  (0.0207682, 4.54298e-4)
  (0.0139686, 3.41713e-4)
  (0.0110998, 3.26436e-4)
  (0.00763247,3.22962e-4)
};
\addlegendentry{$\nu=0.0005$}

\addplot[cyan, mark=triangle*] coordinates {
  (0.0284671, 6.00069e-4)
  (0.0207682, 3.22550e-4)
  (0.0139686, 1.26090e-4)
  (0.0110998, 8.21783e-5)
  (0.00763247,7.08700e-5)
};
\addlegendentry{$\nu=0.0001$}

\addplot[purple, mark=*] coordinates {
  (0.0284671, 5.96474e-4)
  (0.0207682, 3.14999e-4)
  (0.0139686, 1.07785e-4)
  (0.0110998, 6.09849e-5)
  (0.00763247,2.97639e-5)
};
\addlegendentry{$\nu=0$}

\end{loglogaxis}
\end{tikzpicture}
\caption{Convergence of the samplet-based universal Kriging under 
  refinement of
the training set, for increasing noise levels $\nu$.}
\label{fig:uk_convergence}
\end{figure}

\section{Conclusion}\label{sec:conclusion}
We have presented a samplet-based framework for the efficient solution of the
saddle-point systems that arise from conditionally positive definite kernel
approximation and, equivalently, from intrinsic Kriging. The main tools are the
vanishing moment property of samplets and the particular structure of the
associated scaling distributions, which lead to a reduced representation of the
polynomial constraint at the discrete level. This enables an efficient null-space
reduction of the indefinite saddle-point system to a symmetric positive definite
system for the detail coefficients together with a small triangular system for the
polynomial coefficients. Combined with the quasi-sparsity of the samplet-compressed
kernel matrix for asymptotically smooth kernels, the assembly and the storage of
the saddle-point system are of cost $\mathcal{O}(N\log N)$, and the reduced system
becomes amenable to a sparse Cholesky factorization with a fill-in reducing
ordering. For polyharmonic splines in Beppo-Levi spaces we have derived error
estimates for the regularized least-squares approximation that account explicitly
for the observational noise.

The numerical experiments indicate that the proposed approach remains accurate and
scalable for both universal Kriging with generalized covariance matrix and for image warping and volumetric
deformation. In these test cases, the efficient solver translates into substantially
larger problem sizes than those practically accessible with dense direct
solvers.


\bibliographystyle{plain}
\bibliography{literature}
\end{document}